\documentclass[final]{article}

\usepackage{amsmath}
\usepackage{amsfonts}
\usepackage{amsthm}
\usepackage[english]{babel} 

\usepackage{algorithmic}
 \usepackage{amssymb}
\usepackage{graphicx}        
 \usepackage{cite}                            
\usepackage{subfigure}
\usepackage{epstopdf}
\usepackage{epsfig}
\usepackage[table]{xcolor}
\usepackage{subfloat}
\usepackage{float}
\usepackage[thinlines]{easytable}
\usepackage{array}
\allowdisplaybreaks

\newtheorem{remark}{Remark}[section]
\newtheorem{definition}{Definition}[section]

\newtheorem{lemma}{Lemma}[section]
\newtheorem{theorem}{Theorem}[section]
\newtheorem{proposition}{Proposition}[section]
\usepackage[linesnumbered,ruled,vlined]{algorithm2e}

\def\TE{\tilde{\mathcal E}}

\def\OV{\overline V}
\def\mF{\mathcal F}
\def\Vb{v^*}
\def\RR{\mathbb R}
\def\EE{\mathcal E}

\def\SS{\mathbb S}

\def\theta{\vartheta}

\def\diag{{\rm diag}}

\def\argmin{{\rm arg}\!\min}
\def\be{\begin{equation}}
\def\ee{\end{equation}}
\def\bea{\begin{eqnarray}}
\def\eea{\end{eqnarray}}
\def\Pi{P}

\newtheorem{assu}{Assumption}[section]
\newcommand{\mc}[1]{\mathcal{#1}}

\newcommand{\la}{\langle}
\newcommand{\ra}{\rangle}

\title{Anisotropic Diffusion in Consensus-based Optimization on the Sphere}
\author{Massimo Fornasier \footnote{Department of Mathematics, Technical University of Munich, Boltzmannstra{\ss}e 3, 85748 Garching (Munich), Germany
(massimo.fornasier{@}ma.tum.de).}\qquad
Hui Huang \footnote{Department of Mathematics and Statistics, University of Calgary
	(hui.huang1@ucalgary.ca), 2500 University Drive NW
Calgary, AB, Canada.}\qquad
Lorenzo Pareschi\footnote{Department of Mathematics \& Computer Science, University of Ferrara, Via Machiavelli 30, Ferrara, 44121, Italy (lorenzo.pareschi{@}unife.it).}\qquad
Philippe S\"{u}nnen \footnote{Department of Mathematics, Technical University of Munich, Boltzmannstra{\ss}e 3, 85748 Garching (Munich), Germany
	(philippe.suennen{@}ma.tum.de).}
}

\begin{document}
\maketitle

\begin{abstract}
  In this paper we are concerned with the global minimization of a possibly non-smooth and non-convex objective function constrained on the unit hypersphere by means of a multi-agent derivative-free method. The proposed algorithm falls into the class of the recently introduced Consensus-Based Optimization. In fact, agents move on the sphere driven by a drift towards an instantaneous consensus point, which is computed as a convex combination of agent locations, weighted by the cost function according to Laplace's principle, and it represents an approximation to a global minimizer. The dynamics is further perturbed by an anisotropic random vector field to favor exploration. The main results of this paper are about the proof of convergence of the numerical scheme to global minimizers provided conditions of well-preparation of the initial datum. The proof of convergence combines a mean-field limit result with a novel asymptotic analysis, and classical convergence results of numerical methods for SDE. The main innovation with respect to previous work is the introduction of an anisotropic stochastic term, which allows us to ensure the independence of the parameters of the algorithm from the dimension and to scale the method to work in very high dimension. We present several numerical experiments, which show that the algorithm proposed in the present paper is extremely versatile and outperforms previous formulations with isotropic stochastic noise.
\end{abstract}

{\bf Keywords}:   high-dimensional optimization, derivative-free optimization, geometric optimization, consensus-based optimization, anisotropic stochastic Kuramoto-Vicsek model,  Fokker-Planck equations, signal processing and machine learning

\tableofcontents
\section{Introduction}

In this paper we are concerned with the global minimization of a possibly non-convex objective function $\mathcal E:\mathbb R^{d} \to \mathbb R$ constrained on the unit hypersphere $\mathbb S^{d-1}=\{v \in \mathbb R^d: |v|=1\}$
\begin{equation}\label{mainopt}
				v^* \in \arg \min_{v \in \mathbb S^{d-1}} \mathcal E(v).
\end{equation}
We are particularly interested in the case where $\mathcal E$ is a continuous function and its point-wise evaluations are accessible, but it may not be necessarily smooth enough to allow evaluations of its derivatives. Additionally we shall consider the problem of making such optimization feasible in very high dimension.\\
The optimization problem \eqref{mainopt} is ubiquitous in the natural sciences, engineering or computer science. A classical example with applications in moderate dimension $d$ is the Weber problem, where one wishes to find the median barycenter on a three dimensional sphere \cite{NORMANKATZ1980175}. Likewise, a variety of nonlinear optimization problems on a sphere need to be performed over the surface of the Earth in  geophysics, climate modeling, or global navigation \cite{7029667,Chen2018SphericalDA}. High-dimensional man-made optimization problems on the hypersphere appear in machine learning, see, e.g., \cite{fiedler2021stable}, where the authors show that the problem of identification of a generic deep neural network can be reformulated through second order differentiation into an optimization problem over the sphere of the type
$
\arg \min_{v \in \mathbb S^{d-1}} \| P_{\mathcal{W}} (v\otimes v\otimes\dots \otimes v)\|$,
where $\mathcal W$ is a suitable subspace of symmetric tensors, $P_{\mathcal{W}}$ is the orthoprojector onto $\mathcal W$, and $\| \cdot \|$ is the tensor spectral norm. The same type of optimization is used for efficiently computing symmetric rank-$1$ tensor decompositions \cite{kileel2019subspace}. Similarly, finding the largest and smallest Z-eigenvalues of an even order symmetric tensor \cite{QI20051302} is equivalent to calculate the maximum and minimum values of a homogeneous polynomial associated with a tensor on a unit sphere, respectively. The minimization of quartic function over the hyperspheres  is used in the minimization of the empirical risk in phase retrieval problems \cite{ca14,chandra2017phasepack,FHPS2}. Other spherical optimization problems which have non-smooth objectives include the robust subspace detection \cite{lerman15,lerman19} or the sparse principal component analysis (PCA) \cite{absi17}. 

For  applications where the analytic form of the objective function is costly or impossible to access, derivative-free methods in nonlinear optimization have been widely considered \cite{ bandeira12,marazzi02,powell02,zhang14}. One may refer to monographs and reviews \cite{ConnScheVice09,larson_menickelly_wild_2019} for more general theory and literature review on derivative-free methods. However, all these methods are based on an individual agent iteration and may not necessarily come with global convergence guarantees in the case the objective function is non-convex, as in the examples we mentioned above. \\
In this paper we are concerned with derivative-free solutions to tackle global non-convex optimization, which fall into the class of  \textit{metaheuristics}, see \cite{NeldMead65,kennedy2010agent,
poli2007agent,dorigo2005ant,holley1988simulated}. 
Despite the tremendous empirical success of these techniques, it is still quite difficult to provide mathematical guarantees of robust convergence to global minimizers, because of the random component of metaheuristics, which would require to discern the stochastic dependencies. Such analysis is often a very hard task, especially for those methods that combine instantaneous decisions with memory mechanisms. 

Recent work  by Pinnau, Carrillo et al. \cite{PTTM,carrillo2018analytical}  introduced Consensus-based Optimization (CBO) which is a multi-agent derivative-free method defined as instantaneous stochastic and deterministic decisions in order to establish a consensus among agents on the location of the global minimizers within a domain. 
Certainly CBO is a significantly simpler mechanism with respect to more sophisticated metaheuristics, which may include different features including memory of past exploration. Nevertheless, it seems to be powerful and robust enough to tackle many interesting non-convex optimizations of practical relevance also in high-dimensional problems in machine learning \cite{carrillo2019consensus}, and most importantly, it allows for proofs of convergence \cite{carrillo2018analytical,ha2020convergence}. 

By now, CBO methods have been generalized also to optimizations over manifolds \cite{FHPS1,FHPS2,9304325} and several variants have been explored, which use additionally, for instance, personal best information \cite{totzeck2020consensusbased} or connect CBO with other metaheuristic methods such as Particle Swarm Optimization \cite{grassi2020particle}.
In particular in \cite{FHPS1,FHPS2}, we introduced a novel numerical CBO method to solve optimizations on hyperspheres, defined as follows:
generate $V_0^i$, $i=1,\ldots,N$ sample vectors according to $\rho_0 \in \mathbb S^{d-1}$ and iterate for $n=0,1,\dots$
\begin{align}\label{Intro KViso num}
\tilde V^i_{n+1} &\gets V^i_n +\Delta t \lambda P(V_n^i)V_n^{\alpha, \EE} + \sigma |V_n^i - V_n^{\alpha, \EE}| P(V_n^i)\Delta B_n^i \nonumber \\
&\quad \quad \quad \quad -\displaystyle\Delta t\frac{\sigma^2}{2}(V_n^i-V_n^{\alpha, \EE})^2(d-1)V_n^i, \\
V^i_{n+1} &\gets \tilde V^i_{n+1}/|\tilde V^i_{n+1}|, \quad i=1,\ldots,N, \nonumber
\end{align}	
where $\Delta B_n^i$ are independent normal random vectors normally distributed as $\mathcal N(0,\Delta t)$. 
We name this method the {\it isotropic} Kuramoto-Vicsek CBO (KV-CBO) as it is very much inspired by the homogeneous version of the kinetic Kolmogorov-Kuramoto-Vicsek model \cite{demo07,coetal02,Vicsek1995NovelTO}.
As one can notice, this scheme is derivative-free as only point evaluations of $\mathcal E$ are used in the computation of $V_n^{\alpha, \EE}$, which is defined as
\begin{equation}
V_n^{\alpha, \EE} = \frac1{N_\alpha} \sum_{j=1}^N w_\alpha^{\EE}(V^j_n)V^j_n,\qquad N_\alpha = \sum_{j=1}^N w_\alpha^{\EE}(V^j_n),
\label{eq:valpha}
\end{equation}
where $w_\alpha^{\EE}(V^j_n)=\exp(-\alpha\EE(V^j_n))$.
This  iteration corresponds to the discrete time Euler-Maruyama approximation of the Kuramoto-Vicsek (KV) stochastic differential equation system 
\begin{align}\label{Intro KViso}
dV_t^i = \lambda P(V_t^i) v_{\alpha, \EE} (\rho_t^N) dt + \sigma |\mF_t^i(\rho^N)|P(V_t^i)dB_t^i - \frac{\sigma^2}{2}(d-1)(\mF_t^i(\rho^N))^2 \frac{V_t^i}{|V_t^i|^2}dt 
\end{align}
where $\mF_t^i(\rho^N):=V_t^i -  v_{\alpha, \EE} (\rho_t^N)$ with
\begin{equation}
 v_{\alpha, \EE} (\rho_t^N) = \sum_{j=1}^N \frac{V_t^j e^{-\alpha \EE (V_t^j)}}{\sum_{i=1}^N e^{-\alpha \EE(V_t^i)}} = \frac{\int_{\mathbb R^{d}}ve^{-\alpha \EE(v)}d\rho_t^N}{\int_{\mathbb R^{d}}e^{-\alpha \EE(v)}d\rho_t^N} , \quad\mbox{with } \rho_t^N:= \frac{1}{N}\sum_{i=1}^N \delta_{V_t^i}
\end{equation}
and the projection operator $P$ onto the tangent space on the sphere is given by
$P(v)=I-|v|^{-2}(v \otimes v)$,
and satisfies $P(v)v = 0$, and $v\cdot P(v)y=0$ for all $y\in \RR^d$. 
A discussion on the mechanism of the dynamics is extensively provided in \cite{FHPS1,FHPS2}.
\\
The proof of convergence of the method \eqref{Intro KViso num} is based on a three level approximation argument: the discrete time approximation \eqref{Intro KViso num} is shown by standard arguments of numerical approximation of SDE \cite{Platen} to approximate the solution of the  first order stochastic differential equations (SDEs) \eqref{Intro KViso}. Then the large agent limit for $N\to \infty$ is approximated by the solution $\rho_t$ of a deterministic partial differential equation of mean-field type. Finally the large time behavior of the solution $\rho_t$ of such a deterministic PDE to converge to a Dirac delta near a global minimizer can be analyzed by classical calculus. The combination of these three approximations yields the convergence of the method in terms of a quantitative  estimate of the error to a global minimizer \cite{FHPS2}.

\subsection{Scope of the paper and main result}

The scope of the present paper is to introduce a version of the CBO method for optimizations on hyperspheres \eqref{mainopt} that also implements an anisotropic noise and to prove its convergence to global minimizers without explicit dependence of the parameters on the dimension.
Inspired by the work \cite{carrillo2019consensus}, let us introduce an \textit{anisotropic} variant of the \textit{isotropic KV-CBO} from \eqref{Intro KViso}, that is, we replace the diffusion term $\sigma |V_t^i - v_\alpha|P(V_t^i) dB_t^i$ by the anisotropic term 
\begin{equation}\label{aninoise}
\sigma P(V_t^i)D(V_t^i -  v_{\alpha, \EE} (\rho_t^N)) dB_t^i :=\sigma\sum_{k=1}^{d}P(V_t^i)(V_t^i -  v_{\alpha, \EE} (\rho_t^N))_kdB_t^{i(k)}e_k
\end{equation}
and  $B_t^i$ for $t \geq 0$ and $i=1,...,N$ denote $N$ independent standard Brownian motions in $\RR^d$ with components $B_t^{i(k)}$ for $k=1,...,d$, namely,  they are $d$ independent 1-dimensional Brownian motions. The stochastic term \eqref{aninoise} is essentially the projection of the anisotropic noise term of the Euclidean space in which the sphere is embedded onto the tangent space of the sphere. This means that the coordinate direction of the embedding Euclidean space are playing a privileged role in this model and they will influence its dynamics. In particular, cardinal positions on the sphere, e.g., the north pole, are privileged locations for minimizers where the effect of the anisotropy will be maximal and render the independence on the dimension of the optimization algorithm more pronounced. 
The construction of a method which possesses locally fully anisotropic noise uniformly on the sphere would require a moving frame approach, which is not only very difficult to analyze from a theoretical point of view, but it is also extremely computational intensive, especially in high-dimension. 

The anisotropic KV-CBO method takes now the form: generate $V_0^i$, $i=1,\ldots,N$ sample vectors uniformly on $\mathbb S^{d-1}$ and iterate for $n=0,1,\dots$ 
\begin{eqnarray}
\tilde V^i_{n+1} &\gets& V^i_n + \Delta t\lambda P(V_n^i)V_n^{\alpha,\EE} + \sigma  P(V_n^i)D(V_n^i - V_n^{\alpha, \EE})\Delta B_n^i \label{eq:EM} \nonumber \\
&&\ -\Delta t \frac{\sigma^2}{2}\left(|V_n^i-V_n^{\alpha,\EE}|^2+D(V_n^i - V_n^{\alpha, \EE})^2-2\left|D(V_n^i - V_n^{\alpha, \EE})V_n^i\right|^2\right)V^i_n, \label{Intro KViso num3}\\
V^i_{n+1} &\gets& \tilde V^i_{n+1}/|\tilde V^i_{n+1}|, \quad i=1,\ldots,N, \nonumber
\end{eqnarray}
where $\Delta B_n^i$ are independent normal random vectors $\mathcal N(0,\Delta t)$ and $D(V_n^i - V_n^{\alpha, \EE})^2:=\diag ((V_n^i - V_n^{\alpha, \EE})_1^2,\cdots,(V_n^i - V_n^{\alpha, \EE})_d^2)\in\RR^{d\times d}$. The main result of this paper is summarized concisely by the following statement.
\begin{theorem}\label{mainresult001} 
Assume $\EE \in C^2(\SS^{d-1})$ and that for any $v \in \SS^{d-1}$ there exists  a minimizer $v^\star \in \SS^{d-1}$ of $\EE$ (which may depend on $v$) such that  it holds	
\begin{equation}
		|v-v^\star| \leq  C_0|\EE(v)-\underline \EE|^\beta\,,
	\end{equation}
where  $\beta, C_0$ are some positive constants and $\underline \EE:=\inf_{v \in \SS^{d-1}}\EE(v)$. We also denote 
$\overline \EE:=\sup_{v \in \mathbb S^{d-1}}\EE(v)$, $C_{\alpha,\EE}=e^{\alpha (\overline \EE-\underline \EE)}$, and $C_{\sigma}=\frac{\sigma^2}{2}$.
		Additionally for any $\epsilon>0$ assume that the initial datum and parameters are {\it well-prepared} in the sense of Definition \ref{def:wellprep} for a time horizon $T^*>0$  and parameter $\alpha^*>0$ large enough. Then, the iterations $\{V_{n}^i:=V_{\Delta t, n}^i: n=0,\dots,n_{T^*}; i=1\dots N\}$ generated by \eqref{Intro KViso num3}   fulfill the following error estimate
\begin{eqnarray}\label{mainresultXX}
		\mathbb E \left   [\left|\frac{1}{N} \sum_{i=1}^N V_{ n_{T^*}}^i - v^* \right|^2 \right ] 
		&\lesssim& \underbrace{C_1 (\Delta t)^{2m}}_{Discr.\, err.}+ \underbrace{C_2 N^{-1}}_{{
				Mean-field\, lim.}} + \underbrace{C_3 \epsilon^2}_{{Laplace\, princ.}} \,,
		\end{eqnarray}
		 where $m=1/2$ is the order of approximation of the numerical scheme. 
The constant $C_1$ depends linearly on the dimension $d$ and the number of particles $N$, and possibly exponentially on $T^*$ and the parameters $\lambda$ and $\sigma$;  the constant $C_2$ depends linearly on the dimension $d$, polynomially in $C_{\alpha^*,\EE}$, and exponentially in $T^*$; the constant $C_3$ depends on $C_0$ and $\beta$. The convergence is exponential with rate
\begin{equation}\label{rateofconv}
\lambda\theta- 4C_{\alpha^*,\EE}C_{\sigma}>0,
\end{equation}
for a suitable $0<\theta<1$.
\end{theorem}
 A few comments about this result are in order: The quantitative error bound \eqref{mainresultXX} is composed of three terms. The first term is about the approximation error of the numerical scheme. The second is the quantitative estimate of the  mean-field approximation of the large agent limit for $N\to \infty$. The last error estimate is due to the large time behavior of the mean-field approximation and the Laplace's principle.

The smoothness of the objective function $\mathcal E$ is exclusively needed because of our proving technique based on differential calculus to establish well-posedness of SDE, PDE, and large-time behavior. The method \eqref{Intro KViso num3} is effectively derivative-free and it can be used for non-smooth objective functions $\mathcal E$ as no evaluation of derivatives is required for its realization.

The well-preparation of the initial conditions and of the parameters essentially requires the initial distribution $\rho_0$ of the agents to have small variance and be centered not too far from one of the minimizers $v^*$ of $\mathcal E$. 
This suggests that the convergence is local, but for symmetric objective functions $\mathcal E(v)= \mathcal E(-v)$, as in our numerical experiments below, the condition is generically satisfied because of symmetry and therefore the result is essentially of global convergence in these cases. Most importantly, the well-preparation of the parameters do not require their  dependence on the dimension.

\subsection{Proof of the main result and organization of the paper}

The proof of Theorem \ref{mainresult001} follows similar arguments as developed in the papers \cite{FHPS1,FHPS2}, where we analyzed the convergence of the isotropic version of the method. Hence, some of the reasoning will be reported more concisely and we will refer to the corresponding results in \cite{FHPS1,FHPS2} in case no essential innovation is needed to be explained. 
\begin{proof}
The proof of Theorem \ref{mainresult001} goes through the following fundamental steps, which are developed in more detail in Section \ref{sec:wellposednessMF}. 
 The iterative algorithm \eqref{Intro KViso num3} is the discrete-time (Euler-Maruyama) approximation of the SDE system
\begin{align} 
&dV_t^i = \lambda P(V_t^i) v_{\alpha, \EE} (\rho_t^N)dt + \sigma  P(V_t^i)D(\mF_t^i(\rho^N))dB_t^i -\frac{\sigma^2}{2}|\mF_t^i(\rho^N)|^2\frac{V_t^i}{|V_t^i|^2}dt \nonumber \\
&\quad -\frac{\sigma^2}{2}D(\mF_t^i(\rho^N))^2\frac{V_t^i}{|V_t^i|^2}dt +\sigma^2\left|D(\mF_t^i(\rho^N))V_t^i\right|^2\frac{V_t^i}{|V_t^i|^4}dt, \label{sKV}
\end{align}
for $i=1,\cdots,N$, where $\mF_t^i(\rho^N)=V_t^i-v_{\alpha,\EE}(\rho_t^N)$, and $\lambda, \sigma>0$ are suitable drift and diffusion parameters respectively. Its well-posedness is established in Theorem \ref{thmwellposednessofagent}. We also establish in Theorem \ref{thmself} the well-posedness of the an auxiliary self-consistent nonlinear SDE  satisfying
\begin{align}
	d\OV_t &= \lambda P(\OV_t)v_{\alpha,\EE}(\rho_t)dt + \sigma  P(\OV_t)D(\OV_t - v_{\alpha,\EE}(\rho_t))dB_t-\frac{\sigma^2}{2}|\OV_t-v_{\alpha,\EE}(\rho_t)|^2\frac{\OV_t}{|\OV_t|^2}dt\notag\\
	&\quad -\frac{\sigma^2}{2}D(\OV_t-v_{\alpha,\EE}(\rho_t))^2\frac{\OV_t}{|\OV_t|^2}dt+\sigma^2\left|D(\OV_t-v_{\alpha,\EE}(\rho_t))\OV_t\right|^2\frac{\OV_t}{|\OV_t|^4}dt\,,\label{selfprocess}
\end{align}
with the initial data $\overline V_0$ distributed according to $\rho_0\in\mathcal{P}(\SS^{d-1})$ and $\rho_t=\rm{law}(\overline V_t)$, and
$
v_{\alpha,\EE}(\rho_t)  = \frac{\int_{\SS^{d-1}} v \omega_\alpha^\EE(v)\,d\rho_t}{\int_{\SS^{d-1}} \omega_\alpha^\EE(v)\,d\rho_t}.
$
 We further prove by Theorem \ref{thmPDE} that $\rho_t$ solves  the mean-field equation 
\begin{align}\label{PDE}
	&\partial_t \rho_t= \lambda \nabla_{\SS^{d-1}} \cdot ((\langle v_{\alpha, \EE }(\rho_t), v \rangle v - v_{\alpha,\EE}(\rho_t) )\rho_t)+\frac{\sigma^2}{2}\sum_{i=1}^d\partial_{v_i^{\SS}}^2((v-v_{\alpha,\EE}(\rho_t))_i^2\rho_t)\\
	&-\frac{\sigma^2}{2}(d-2)\nabla_{\SS^{d-1}} \cdot (D(v-v_{\alpha, \EE }(\rho_t))^2v\rho_t)+\frac{\sigma^2}{2}(d-2)(d-1)|D(v-v_{\alpha, \EE }(\rho_t))v|^2\rho_t\notag \\
	&
	-\frac{\sigma^2}{2}(d-1)\sum_{i=1}^d\partial_{v_i^{\SS}}((v-v_{\alpha,\EE}(\rho_t))_i^2\rho_t)v_i,\quad t>0,~v\in\SS^{d-1}\,,\notag
	\end{align}
with the initial data $\rho_0\in\mathcal {P}(\SS^{d-1})$. Then we show in Theorem \ref{thmmean}
\begin{equation*}
	\sup_{t\in[0,T]}\sup_{i=1,\dots,N}\mathbb E |V_t^i-\OV_t^i|^2 \lesssim N^{-1}\to 0,
	\end{equation*}
	as $N \to \infty$, where $((\OV_t^i)_{t\geq 0 })_{i=1,\dots,N}$ are $N$ identical copies of solutions to \eqref{selfprocess}. So they are i.i.d. with the common law $\rho_t$ satisfying the mean-field PDE \eqref{PDE}.
	The mean-field limit will be achieved through the coupling method \cite{sznitman1991topics,fetecau2019propagation}.
 Finally we investigate in Theorem \ref{thm:mainresult} the large time behavior of $\rho_t$: Let $\epsilon>0$ and assume that the initial datum and parameters are in the sense of Definition \ref{def:wellprep} for a time horizon $T^*>0$  and parameter $\alpha^*>0$ large enough. Then $E(\rho_{T^*})= \int_{\mathbb S^{d-1}} v d\rho_{T^\star}(v)$  well approximates a minimizer $v^*$ of $\EE$, 
\begin{equation*}\label{locest}
	|E(\rho_{T^*})-v^*|\leq  \epsilon\,.
	\end{equation*}
	 The concluding step takes into account all the approximation results. By using the  order $m=1/2$ numerical scheme \eqref{Intro KViso num3} (see, e.g., \cite{doi:10.1137/S0036142901389530}),  Theorem \ref{thmmean}  and Theorem \ref{thm:mainresult}, and combining them 
 by using triangle-like and Jensen inequalities we obtain
\begin{align*}\label{mr}
 &\mathbb E \left   [\left|\frac{1}{N} \sum_{i=1}^N V_{\Delta t, n_{T^*}}^i - v^* \right|^2 \right ] \nonumber 
 \lesssim \mathbb E \left   [\left|\frac{1}{N} \sum_{i=1}^N( V_{\Delta t,n_{ T^*}}^i -  V_{T^*}^i) \right|^2 \right ]  +  \mathbb E \left   [\left|\frac{1}{N} \sum_{i=1}^N (V_{T^*}^i -  \OV_{T^*}^i) \right|^2 \right ]  \notag\\
 &+  \mathbb E \left   [\left|\frac{1}{N} \sum_{i=1}^N \OV_{T^*}^i - E(\rho_{T^*})  \right|^2 \right ] + |E(\rho_{T^*}) - v^*|^2 
 \lesssim (\Delta t)^{2m} + N^{-1} + \epsilon^2\,. 
\end{align*}
The dependence of the constants on model parameters, dimension, and number of agents can be read from the respective literature and proofs in this paper and are explicitly mentioned in the statement of the theorem.
This concludes the proof.
\end{proof}

The theoretical results are illustrated and validated by extensive numerical experiments in Section \ref{sec:num}. There we show the actual numerical implementation of the method and we discuss possible relatively simple algorithmic improvements, which allow for computationally efficient realizations, such as numerically stable implementation of large choices of the parameter $\alpha$ or variance based discarding of agents to reduce the complexity.

\section{Well-posedness, mean-field limit, and large time asymptotics} \label{sec:wellposednessMF}

This section focuses on proving the well-posedness for the agent system \eqref{sKV}, the mean-field dynamic \eqref{selfprocess} and the mean-field PDE \eqref{PDE}. We also verify the mean-field limit of the the agent system \eqref{sKV} towards the nonlinear PDE \eqref{PDE}. We conclude with the analysis of the asymptotic behavior of the solution to the PDE \eqref{PDE}.

\subsection{Well-posedness of the SDE}

We assume the objective function $\EE$ is \textit{locally Lipschitz} continuous.
To begin we shall assume that the agent system \eqref{sKV} is assumed to evolve in the whole space $\RR^d$ instead of on the sphere $\SS^{d-1}$ directly. We choose this embedding because it provides an explicit and computable representation
of the system and it allows for a global description. The difficulty in showing first  the well-posedness of \eqref{sKV} in the ambient space $\RR^d$  is that the projection $P(V_t^{i})$ is not defined for $V_t^{i}=0$ (singularity), and $\frac{V_t^{i}}{|V^{i}_t|^2}, \frac{V_t^{i}}{|V^{i}_t|^4}$ is unbounded for $V_t^{i}=0$ (blow-up) . In order to overcome this problem, we regularize the diffusion and drift coefficients, that is, we replace them with appropriate functions $P_1$,  $P_2$ and $P_3$ respectively: let $P_1$ be a $d\times d$ matrix valued map on $\RR^d$ with bounded derivatives of all orders such that $P_1(v)= P(v)$ for all $|v|\geq\frac{1}{2}$, and $P_2,P_3$ be a $\RR^d$ valued map on $\RR^d$, again with bounded derivatives of all orders, such that 
$P_2(v)=\frac{v}{|v|^2}$ and  $P_3(v)=\frac{v}{|v|^4}$ if $|v|\geq \frac{1}{2}$. For later use, let us denote $[N] = \{1,\dots,N\}$.
Additionally, we regularize the locally Lipschitz continuous function $\EE$: Let us introduce $\TE$ satisfying the following assumptions: 
\begin{assu}\label{asum}
	The regularized extension function $\TE:\RR^d\rightarrow \RR$ is globally Lipschitz continuous and satisfies the properties
	\begin{itemize}
		\item[1.]  $\TE(v)=\EE(v)$ when $|v|\leq \frac{3}{2}$, and $\TE(v)=0$ when  $|v|\geq 2$;
		\item[2.]  There exists some $L>0$ such that $\TE(v)-\TE(u)\leq L|v-u|$ for all $u,v\in\RR^d$;
		\item[3.]	$-\infty < \underline{\TE}:=\inf \TE \leq \TE\leq \sup \TE=:\overline \TE < + \infty$\,.
	\end{itemize}
\end{assu}

Given such $P_1$, $P_2$, $P_3$ and $\TE$ satisfying Assumption \ref{asum}, we introduce the following regularized agent system
\begin{align}
dV_t^i &= \lambda P_1(V_t^i)v_{\alpha,\TE}(\rho_t^N)dt + \sigma  P_1(V_t^i)D(\widetilde{\mF}_t^i(\rho^N))dB_t^i-\frac{\sigma^2}{2}|\widetilde{\mF}_t^i(\rho^N)|^2P_2(V_t^i)dt\\
&\quad -\frac{\sigma^2}{2}D(\widetilde{\mF}_t^i(\rho^N))^2P_2(V_t^i)dt+\sigma^2\left|D(\widetilde{\mF}_t^i(\rho^N))V_t^i\right|^2P_3(V_t^i)dt \label{RSKV}
\end{align}
for $i \in [N]$,  where $\widetilde{\mF}_t^i(\rho^N):=V_t^i-v_{\alpha,\TE}(\rho_t^N)$ with
$
v_{\alpha,\TE}(\rho_t^N)=\frac{\int_{\mathbb R^{d}}v\omega_\alpha^{\TE} (v)d\rho_t^N}{\int_{\mathbb R^{d}}\omega_\alpha^{\TE}(v)d\rho_t^N}$ and $ \omega_\alpha^{\TE} (v)=e^{-\alpha{\TE} (v)}.$

Our first theorem states the well-posedness for the interacting agent system \eqref{sKV}:
\begin{theorem}\label{thmwellposednessofagent}
	Let $\rho_0$ be a probability measure on $\SS^{d-1}$ and, for every $N\in\mathbb{N}$, $(V_0^i)_{i \in [N]}$ be $N$ i.i.d. random variables with the common law $\rho_0$.
	For every $N\in\mathbb{N}$, there exists a path-wise unique strong solution $((V_t^i)_{t\geq 0})_{i \in [N]}$ to the agent system \eqref{sKV} with the initial data $(V_0^i)_{i \in [N]}$. Moreover it holds that $V_t^i\in \SS^{d-1}$ for all $i \in [N]$ and any $t>0$.
\end{theorem}
\begin{proof}
	Given $P_1$, $P_2$, $P_3$ and $\TE$, the SDE \eqref{RSKV} has locally Lipschitz continuous coefficients according to \cite[Lemma 2.1]{FHPS1}, so it admits a path-wise unique local strong solution by standard SDE well-posedness result \cite[Chap. 5, Theorem 3.1]{durrett2018stochastic}. Moreover, it follows from It\^{o}'s formula that as long as $|V_t^i|\geq 1/2$, it holds
\begin{align}\label{dVti2=0}
	d|V_t^i|^2&= 2\lambda V_t^i\cdot P(V_t^i)v_{\alpha,\TE}(\rho_t^N)dt +2 \sigma  V_t^i\cdot P(V_t^i)D(\widetilde{\mF}_t^i(\rho^N))dB_t^i-\sigma^2|\widetilde{\mF}_t^i(\rho^N)|^2dt \notag\\
	&\quad -\sigma^2\frac{V_t^i\cdot D(\widetilde{\mF}_t^i(\rho^N))^2V_t^i}{|V_t^i|^2} +2\sigma^2\frac{\left|D(\widetilde{\mF}_t^i(\rho^N))V_t^i\right|^2}{|V_t^i|^2}dt  \notag\\
	&\quad +\sum_{\ell=1}^d\sigma^2\left((\widetilde{\mF}_t^i(\rho^N))_\ell^2+\sum_{k=1}^d(\widetilde{\mF}_t^i(\rho^N))_k^2\frac{
		(V_t^{i(\ell)})^2(V_t^{i(k)})^2}{|V_t^i|^4}-2 (\widetilde{\mF}_t^i(\rho^N))_\ell^2\frac{(V_t^{i(\ell)})^2}{|V_t^i|^2}\right)dt\notag\\
	&=-\sigma^2|\widetilde{\mF}_t^i(\rho^N)|^2dt +\sigma^2\frac{\left|D(\widetilde{\mF}_t^i(\rho^N))V_t^i\right|^2}{|V_t^i|^2}dt+\sum_{\ell=1}^d\sigma^2(\widetilde{\mF}_t^i(\rho^N))_\ell^2dt\notag\\
	&\quad+\sum_{\ell=1}^d\sum_{k=1}^d\sigma^2(\widetilde{\mF}_t^i(\rho^N))_k^2\frac{(V_t^{i(\ell)})^2(V_t^{i(k)})^2}{|V_t^i|^4}dt -2\sum_{\ell=1}^d\sigma^2(\widetilde{\mF}_t^i(\rho^N))_\ell^2\frac{(V_t^{i(\ell)})^2}{|V_t^i|^2}dt \notag\\
	&=-\sigma^2|\widetilde{\mF}_t^i(\rho^N)|^2dt +\sigma^2\frac{\left|D(\widetilde{\mF}_t^i(\rho^N))V_t^i\right|^2}{|V_t^i|^2}dt+\sigma^2|\widetilde{\mF}_t^i(\rho^N)|^2dt\notag\\
	&\quad+\sigma^2\frac{\left|D(\widetilde{\mF}_t^i(\rho^N))V_t^i\right|^2}{|V_t^i|^2}dt
	-2\sigma^2\frac{\left|D(\widetilde{\mF}_t^i(\rho^N))V_t^i\right|^2}{|V_t^i|^2}dt=0\,,
	\end{align}
	where $\widetilde{\mF}_t^i(\rho^N):=V_t^i-v_{\alpha,\TE}(\rho_t^N)$, $V_t^{i(k)}$ is the $k$-th component of $V_t^i$, and  we have used  the fact
 $V_t^i\cdot D(\widetilde{\mF}_t^i(\rho^N))^2V_t^i=|D(\widetilde{\mF}_t^i(\rho^N))V_t^i|^2$
	and $P(V_t^i)V_t^i = 0$ in the second equality. 
	Hence $|V_t^i|=|V_0^i|=1$ for all $t>0$, which ensures that the solution keeps bounded at any finite time, hence we have a global solution. Since all $V_t^i$ have norm 1, the solution to the regularized system \eqref{RSKV} is a solution to \eqref{sKV}, which provides the global existence of solutions to \eqref{sKV}. 
	
	To show path-wise uniqueness let us consider two solutions to \eqref{sKV} for the same initial distribution and Brownian motion. According to the above argument these two solutions stay on the sphere for any $t\geq 0$, hence they are solutions to the regularized system \eqref{RSKV}, whose solutions are path-wise unique due to the locally Lipschitz continuous coefficients.  Hence we have uniqueness for solutions to \eqref{sKV}.
\end{proof}

The following theorem states the well-posedness for the nonlinear mean-field dynamic \eqref{selfprocess}.
\begin{theorem}\label{thmself}
	Let $\TE$ satisfy Assumption \ref{asum}. For any $T>0$, there exists a unique process $\overline V\in \mc{C}([0,T],\RR^{d})$ satisfying the nonlinear SDE \eqref{selfprocess}
	for any initial data $\OV_0\in \SS^{d-1}$ distributed according to $\rho_0\in \mc{P}(\SS^{d-1})$.
	Moreover $\overline V_t\in\SS^{d-1}$ for all $t\in[0,T]$.
\end{theorem}
\begin{proof}
	The proof can be done similarly as in the proof of \cite[Theorem 2.2]{FHPS1}, so we only provide a sketch here. For  any  given $\xi\in \mc{C}([0,T],\RR^d)$,  a distribution $\rho_0$ on $\SS^{d-1}$ and $\overline{V}_0$ with law $\rho_0$, we can uniquely solve the SDE
\begin{align*}
	d\OV_t &= \lambda P_1(\OV_t)\xi_tdt + \sigma  P_1(\OV_t)D(\OV_t - \xi_t)dB_t-\frac{\sigma^2}{2}|\OV_t-\xi_t|^2P_2(\OV_t)dt\notag\\
	&\quad -\frac{\sigma^2}{2}D(\OV_t-\xi_t)^2P_2(\OV_t)dt+\sigma^2\left|D(\OV_t-\xi_t)\OV_t\right|^2P_3(\OV_t)dt\,,
	\end{align*}
	and obtain the solution $\overline V_t\in \SS^{d-1}$ for all time.  This introduces $\rho_t=\rm{law}(\OV_t)$ and $\rho\in \mc{C}([0,T],\mc{P}_c(\RR^d))$. Setting $\mc{T}\xi:=v_{\alpha,\TE}(\rho)\in \mc{C}([0,T],\RR^d)$ we define the map
\begin{equation}
	\mc{T}: \mc{C}([0,T],\RR^d)\rightarrow \mc{C}([0,T],\RR^d),\quad  \xi\mapsto \mc{T}(\xi):=v_{\alpha,\TE}(\rho)\,
	\end{equation}
	Then we   apply  Leray-Schauder fixed point theorem to $\mc{T}$, see, e.g., \cite[Chapter 10]{gilbarg2015elliptic}, which provides a  solution to the regularized version of \eqref{selfprocess} :
\begin{align}\label{Rnonlinear}
	d\OV_t &= \lambda P_1(\OV_t)v_{\alpha,\TE}(\rho_t)dt + \sigma  P_1(\OV_t)D(\OV_t - v_{\alpha,\TE}(\rho_t))dB_t-\frac{\sigma^2}{2}|\OV_t-v_{\alpha,\TE}(\rho_t)|^2P_2(\OV_t)dt\notag\\
	&\quad -\frac{\sigma^2}{2}D(\OV_t-v_{\alpha,\TE}(\rho_t))^2P_2(\OV_t)dt+\sigma^2\left|D(\OV_t-v_{\alpha,\TE}(\rho_t))\OV_t\right|^2P_3(\OV_t)dt
	\end{align}
	with $\mbox{law}(\OV_t)=\rho_t$.  We can also easily obtain the uniqueness as the \textit{Step 4} in the proof of \cite[Theorem 2.2]{FHPS1}.
	
	Following the same argument as in \eqref{dVti2=0}, we can easily verify that $|\OV_t|=|\OV_0|=1$ for all $t\in[0,T]$.
	Similar to Theorem \ref{thmwellposednessofagent}, the unique solution to the regularized SDE \eqref{Rnonlinear} obtained through the fixed point theorem is also the unique solution to the nonlinear SDE  \eqref{selfprocess} due to the fact that  $|\OV_t|=1$ for all $t\in[0,T].$
\end{proof}

\subsection{Well-posedness of the PDE}

We prove  the well-posedness for the PDE \eqref{PDE} in the next theorem. Let us first recall some properties of the gradient operator for functions on the sphere and its calculus.
The operator $\nabla_{\SS^{d-1}} =(\partial_{v_1^{\SS}},\cdots, \partial_{v_d^{\SS}})$  denotes the gradient operator on the sphere $\SS^{d-1}$, which satisfies
\begin{equation}
\int_{\SS^{d-1}}\nabla_{\SS^{d-1}} f(v) dv=(d-1)\int_{\SS^{d-1}}vf(v) dv
\end{equation}
and
\begin{equation}
\int_{\SS^{d-1}}f(v)\nabla_{\SS^{d-1}} \cdot A(v)dv=-\int_{\SS^{d-1}}A(v)\cdot \nabla_{\SS^{d-1}} f(v)dv+(d-1)\int_{\SS^{d-1}}A(v)\cdot vf(v)dv
\end{equation}
for regular function $f:~\SS^{d-1}\to \RR$ and regular vector field $A:~\SS^{d-1}\to \RR^{d}$ (not necessary tangent), see for example \cite{frouvelle2012dynamics}. 

\begin{theorem}\label{thmPDE}
	Let $(\overline{V}_t)_{0\leq t\leq T}$ be the unique solution obtained in Theorem \ref{thmself} up to any time $T>0$, and  denote $\rho_t$ as the law of $\OV_t$, which is concentrated on the sphere $\SS^{d-1}$. Then the restriction of $\rho_t$ on the sphere is the unique solution to the nonlinear PDE \eqref{PDE}.
\end{theorem}
\begin{proof}
	Let $(\overline{V}_t)_{0\leq t\leq T}$ be the unique solution to \eqref{selfprocess} obtained in the last theorem with the initial data $\OV_0$ distributed according to $\rho_0\in \mc P(\SS^{d-1})$. For any $\varphi\in C_c^\infty(\RR^d)$, it follows from It\^{o}'s formula  that
\begin{align*}
	d \varphi(\OV_t)
	&=\nabla\varphi(\OV_t)\cdot \bigg(\lambda\left (I-\OV_t \OV_t^T\right)v_{\alpha,\EE}(\rho_t) -\frac{\sigma^2}{2}|\overline{\mF}_t(\rho)|^2\OV_t-\frac{\sigma^2}{2}D(\overline{\mF}_t(\rho))^2\OV_t\\
	&\quad+\sigma^2\left|D(\overline{\mF}_t(\rho))\OV_t\right|^2\OV_t\bigg)dt+\sigma \nabla\varphi(\OV_t)\cdot P(\OV_t)D(\overline{\mF}_t(\rho))dB_t\notag\\
	&\quad+\frac{\sigma^2}{2}\sum_{i}^d(\partial_{v_i}^2\varphi)(\OV_t-v_{\alpha,\EE})_i^2 +\frac{\sigma^2}{2}\sum_{i,j=1}^d\partial_{v_i,v_j}^2\varphi\bigg[-(\OV_t-v_{\alpha,\EE})_i^2\OV_t^{(j)}\OV_t^{(i)}\\
	&\qquad\qquad-(\OV_t-v_{\alpha,\EE})_j^2\OV_t^{(j)}\OV_t^{(i)}+\OV_t^{(j)}\OV_t^{(i)}\left|D(\overline{\mF}_t(\rho))\OV_t\right|^2\bigg]dt\,,
	\end{align*}
	where $\overline{\mF}_t(\rho):=\OV_t-v_{\alpha,\EE}(\rho_t)$, and
	we have used $|\OV_t|^2=1$. Taking expectation on both sides of above identity, we show that  the law $\rho_t$ of $\OV_t$ as a measure on $\RR^d$ satisfies
\begin{align}\label{wholeweak}
	&\frac{d}{dt}\int_{\RR^d}\varphi(v)d\rho_t(v)=\int_{\RR^d}\nabla\varphi(v)\cdot \bigg(\lambda(I-vv^T)v_{\alpha,\EE}(\rho_t) -\frac{\sigma^2}{2}|v-v_{\alpha,\EE}(\rho_t)|^2v\notag\\
	&\qquad-\frac{\sigma^2}{2}D(v-v_{\alpha,\EE}(\rho_t))^2v+\sigma^2\left|D(v-v_{\alpha,\EE}(\rho_t))v\right|^2v\bigg)d\rho_t(v)
	\notag\\
	&\quad+\int_{\RR^d}\frac{\sigma^2}{2}\sum_{i}^d(\partial_{v_i}^2\varphi)(v-v_{\alpha,\EE})_i^2d\rho_t(v)+\frac{\sigma^2}{2}\sum_{i,j=1}^d\partial_{v_i,v_j}^2\varphi\bigg[-(v-v_{\alpha,\EE})_i^2v_jv_i\notag\\
	&\qquad\qquad-(v-v_{\alpha,\EE})_j^2v_jv_i+v_jv_i\left|D(v-v_{\alpha,\EE}(\rho_t))v\right|^2\bigg]dt\,.
	\end{align}
	As we have proved that $|\OV_t|^2=1$, we have $\operatorname{supp}(\rho_t)\subset \SS^{d-1}$  for any $t$. Let us now define the restriction $\mu_t$ of $\rho_t$ on $\SS^{d-1}$ by
\begin{equation}
	\int_{\SS^{d-1}}\Phi(v)d \mu_t(v)=\int_{\RR^d}\varphi(v)d\rho_t(v)
	\end{equation}
	for all continuous maps $\Phi\in \mc{C}(\SS^{d-1})$, where $\varphi\in \mc{C}_b(\RR^d)$ equals $\Phi$ on $\SS^{d-1}$. Let now $\Phi\in \mc{C}^\infty(\SS^{d-1})$ and define
	a function $\varphi\in \mc{C}_c^\infty(\RR^d)$ such that 
\begin{equation}
	\varphi(v)= \Phi\left(\frac{v}{|v|}\right)\quad  \mbox{ for all } \frac{1}{2}\leq |v| \leq 2\,.
	\end{equation}
	Then $\varphi$ defined above  is $0$-homogeneous in $v$ in the annulus $1/2\leq |v|\leq 2$, so that $\nabla\varphi(v)\cdot v=0$ for all $v$ in the support of $\rho_t$. Hence,
\begin{align*}
	&\frac{d}{dt}\int_{\SS^{d-1}}\Phi(v)d\mu_t(v)=\frac{d}{dt}\int_{\RR^d}\varphi(v)d\rho_t(v)=\int_{\RR^d}\frac{\sigma^2}{2}\sum_{i}^d(\partial_{v_i}^2\varphi)(v-v_{\alpha,\EE})_i^2d\rho_t(v)\notag\\
	&+\int_{\RR^d}\nabla\varphi(v)\cdot \bigg(\lambda(I-vv^T)v_{\alpha,\EE}(\rho_t)-\frac{\sigma^2}{2}D(v-v_{\alpha,\EE}(\rho_t))^2v\bigg)d\rho_t(v)\,.
	\end{align*}
	Notice that $\nabla_{\SS^{d-1}} \Phi(\omega) = \nabla \varphi(\omega)$ for all $\omega \in \SS^{d-1}$. Therefore
\begin{align*}
	&\frac{d}{dt}\int_{\SS^{d-1}}\Phi(v)d\mu_t(v)=\int_{\SS^{d-1}}\frac{\sigma^2}{2}\sum_{i}^d(\partial_{v_i^\SS}^2\Phi)(v-v_{\alpha,\EE})_i^2d\mu_t(v)\\
	&+\int_{\SS^{d-1}}\nabla_{\SS^{d-1}}\Phi(v)\cdot \bigg(\lambda(I-vv^T)v_{\alpha,\EE}(\mu_t)-\frac{\sigma^2}{2}D(v-v_{\alpha,\EE}(\mu_t))^2v\bigg)d\mu_t(v)\,.
	\end{align*}
	where
$
	\omega_{\alpha,\EE}(\mu_t) = \frac{\int_{\SS^{d-1}} \omega  e^{-\alpha \EE(\omega)}\,d \mu_t}{\int_{\SS^{d-1}}e^{-\alpha \EE(\omega)}\,d \mu_t}.
$
	Thus we obtain a weak solution $\mu$ to the PDE \eqref{PDE}. 
	
	As for the uniqueness, it can be derived from the uniqueness of the solution $(\overline{V}_t)_{0\leq t\leq T}$ to the nonlinear SDE \eqref{selfprocess}. We refer to \cite[Section 2.3]{FHPS1} for more details.
\end{proof}

\subsection{Mean-field limit}

The well-posedness of  \eqref{sKV}, \eqref{PDE}, and \eqref{selfprocess} obtained above provides all the ingredients we need for the mean-field limit. Let $((\OV_t^i)_{t\geq 0})_{i \in [N]}$ be $N$ independent copies of solutions to \eqref{selfprocess}. They are i.i.d. with the same distribution $\rho_t$. Assume that $((V_t^i)_{t\geq 0})_{i \in [N]}$ is the solution to the agent system \eqref{sKV}. Since  $|\OV_t^i|=|V_t^i|=1$ for all $i$ and $t$, $((\OV_t^i)_{t\geq 0})_{i \in [N]}$ and $((V_t^i)_{t\geq 0})_{i \in [N]}$ are solutions to the corresponding regularized systems \eqref{Rnonlinear} and \eqref{RSKV} respectively. We denote below by $\overline \rho_t^N = \frac{1}{N} \sum_{j=1}^N \delta_{\OV_t^j}$, $\rho_t = \rm{law}(\OV_t)$ and $C_{\alpha,\TE}=e^{\alpha(\overline{\TE}-\underline{\TE})}$.

The mean-field limit states that the i.i.d. mean-field dynamics $((\OV_t^i)_{t\geq 0})_{i \in [N]}$ can well approximate the interacting agent system $((V_t^i)_{t\geq 0})_{i \in [N]}$ in the following sense:
\begin{theorem}[Mean-field limit]\label{thmmean}
	For any $T>0$, under the Assumption \ref{asum}, let $((V_t^i)_{t\in [0,T]})_{i \in [N]}$  and $((\OV_t^i)_{t\in [0,T]})_{i \in [N]}$  be respective solutions to  \eqref{sKV} and \eqref{selfprocess} up to time $T$ with the same initial data $V_0^i=\OV_0^i$ and same Brownian motions $B_t^i$. Then there exists a constant $C>0$ depending only on $\lambda,\alpha$, $d,\sigma$, $\|\nabla P_1\|_\infty$, $\|P_1\|_\infty$,$\|\nabla P_2\|_\infty$, $\|P_2\|_\infty$, $\|\nabla P_3\|_\infty$, $\|P_3\|_\infty,L$ and $C_{\alpha,\TE}$, such that
\begin{equation}\label{thmmeaneq}
	\sup_{i=1,\cdots,N}\mathbb{E}[|V_t^i-\OV_t^i|^2]\leq C Te^{CT}\frac{1}{N}\,,
	\end{equation}
	holds for all $0\leq t\leq T$.
\end{theorem}
\begin{proof}
	We only provide a sketch of the proof here, since it is almost the same as the proof of \cite[Theorem 3.1]{FHPS1}. Notice that $((\OV_t^i)_{t\geq 0})_{i \in [N]}$  and $((V_t^i)_{t\geq 0})_{i \in [N]}$  are also solutions to the corresponding regularized systems \eqref{Rnonlinear} and \eqref{RSKV} respectively. 
	We apply It\^{o}'s formula to $d(V_t^i-\OV_t^i)^2$  and take expectation on both sides, then it is easy to obtain that
\begin{align*}
	&\mathbb{E}[|V_t^i-\OV_t^i|^2]\\
	\leq& \mathbb{E}[|V_0^i-\OV_0^i|^2]+C\int_0^t\sup_{i=1,\cdots,N}\mathbb{E}[|V_s^i-\OV_s^i|^2]ds+C\int_0^t\mathbb{E}[|v_{\alpha,\TE}(\overline \rho_s^N)-v_{\alpha,\TE}(\rho_s)|^2] ds\\
	\leq&\mathbb{E}[|V_0^i-\OV_0^i|^2]+C\int_0^t\sup_{i=1,\cdots,N}\mathbb{E}[|V_s^i-\OV_s^i|^2]ds+CT\frac{1}{N}\,,
	\end{align*}
	where  $C>0$ depends only on $\lambda,\alpha, d,\sigma,\|\nabla P_1\|_\infty, \|P_1\|_\infty$,$\|\nabla P_2\|_\infty$, $\|P_2\|_\infty$,$\|\nabla P_3\|_\infty$, $\|P_3\|_\infty$,$L$ and $C_{\alpha,\TE}$. Here we have used  the large deviation bound {\small $$\sup_{t\in[0,T]}\mathbb{E}\left[|v_{\alpha,\TE}(\overline \rho_t^N)-v_{\alpha,\TE}(\rho_t)|^2\right]\leq C N^{-1}$$}
	 from \cite[Lemma 3.1]{FHPS1}. Applying Gronwall's inequality with $\mathbb{E}[|V_0^i-\OV_0^i|^2]=0$, one concludes \eqref{thmmeaneq}.
\end{proof}

\subsection{Global optimization guarantees}\label{analsec}

In this section, we address the convergence of the stochastic Kuramoto-Vicsek agent  system \eqref{sKV} to global minimizers of some cost function $\EE$ over the sphere $\SS^{d-1}$.
We now define the expectation and variance of $\rho_t$ as
\begin{equation}
E(\rho_t):=\int_{\SS^{d-1}} v d \rho_t (v) \quad V(\rho_t):=\frac{1}{2}\int_{\SS^{d-1}} |v-E(\rho_t)|^2 d\rho_t (v).
\end{equation}
A simple computation yields $2V(\rho_t)=1-E(\rho_t)^2$. In particular,  as soon as $V(\rho_t)$ is small, one has $E(\rho_t)^2 \approx 1$.
Since  $E(\rho_t)=\mathbb{E}[\OV_t]$, it follows from \eqref{selfprocess} that
\begin{equation*}
\frac{d}{dt} E(\rho_t) = -\int_{\SS^{d-1}} \eta_td\rho_t - \int_{\SS^{d-1}}\left(\frac{\sigma^2}{2}(\mc{G}_t(\rho))^2+\frac{\sigma^2}{2}D(\mc{G}_t(\rho))^2 -  \sigma^2\left|D(\mc{G}_t(\rho))v\right|^2\right)vd\rho_t\,,
\end{equation*}
where $\mc{G}_t(\rho):=v-v_{\alpha,\EE}(\rho_t)$ and
$\eta_t:=\lambda \langle v_{\alpha,\EE}(\rho_t),v\rangle v-\lambda v_{\alpha,\EE}(\rho_t)\in \mathbb{R}^d\,.$

Throughout this section, the locally Lipschitiz objective function $\EE$ satisfies the following additional properties
\begin{assu}\label{assumas}\quad
	\begin{itemize}
		\item [1.]  $\EE\in \mathcal{C}^2(\RR^d)$ obtains its global minimum value on the sphere;
		\item[2.] For $v\in \SS^{d-1}$, it holds $0\leq \underline{\EE}:=\inf\limits_{v\in\SS^{d-1}} \EE \leq \EE(v) \leq  \sup\limits_{v\in\SS^{d-1}}  \EE=:\overline \EE < \infty$;
		\item[3.] $\|\nabla \EE\|_\infty\leq c_1$  and $\|\nabla^2 \EE \|_\infty \leq c_2$ for all $v\in\SS^{d-1}$;
		\item[5.]  $\det(\nabla^2\EE(v^*))>0$ for any minimizer $v^*\in \SS^{d-1}$;
		\item [6.] For any $v \in \SS^{d-1}$ 
		there exists  a minimizer $v^*\in \SS^{d-1}$ of $\EE$ (which may depend on $v$) such that  it holds 
$
		|v-v^\ast| \leq  C_0|\EE(v)-\underline \EE|^\beta\,,
$
		where  $\beta, C_0$ are some positive constants.
	\end{itemize}
\end{assu}

Below we denote $C_{\alpha,\EE}:=e^{\alpha(\overline{\EE}-\underline{\EE})}$, $\varepsilon_\alpha:=
O(\frac{1}{\alpha})$ and $C_{\sigma}:=\frac{\sigma^2}{2}$.  The notation $O(\frac{1}{\alpha})$ stands for the fact that there exists some constant $C_1$ depending  only on  $d, 1/\rho_0(v_*)$ and $\det(\nabla^2\EE(v^*))$, such that $|O(\frac{1}{\alpha})|\leq C_1\frac{1}{\alpha}$ holds for $\alpha$ sufficiently large.
\begin{definition}\label{def:wellprep}
	For any given $T>0$, we say that the initial datum and the parameters are well-prepared if $\rho_0\in \mc{P}_{ac}(\SS^{d-1})\cap L^2(\SS^{d-1})$, and parameters $V(\rho_0)$, $\lambda$, $d$,  $\alpha$, $0<\delta\ll 1$ satisfy
\begin{align}
	&C_{\alpha,\EE}^{2\max \{1, \beta\}}\left(V(\rho_0)+\frac{\lambda C_T}{\lambda\theta- 16C_{\sigma}C_{\alpha,\EE}}\delta^{\frac{d-2}{4}} \right)^{\frac{1}{2}\min\{1,\beta\}} +\varepsilon_\alpha^\beta<\frac{\delta-\theta}{C^\ast}\,; \label{wellprep}\\
	&V(\rho_0)+\frac{\lambda C_T}{\lambda\theta- 16C_{\sigma}C_{\alpha,\EE}}\delta^{\frac{d-2}{4}}\leq \min\left\{\frac{\| \omega_\EE^\alpha\|_{L^1(\rho_0)}^2}{T},\frac{\| \omega_\EE^\alpha\|_{L^1(\rho_0)}^4}{T\lambda^2},\frac{3}{8} \right\}
	\end{align}
	and for any $0<\theta<\delta$
\begin{equation}\label{lamsig}
	\lambda\theta- 16C_{\sigma}C_{\alpha,\EE}>0\,,
	\end{equation}
	where $C_T$ is a constant depending only on $\lambda$, $\sigma$, $T$ and $\|\rho_0\|_2$, and  $C^\ast>0$ is a constant depending only on $c_1,c_2,\beta, C_0$ ($c_1,c_2,\beta,C_0$ are used in Assumption \ref{assumas}).  Both $C_T$ and $C^*$ need to be subsumed from the proof of Proposition \ref{mainp} and they are both dimension independent.
\end{definition}
\begin{remark}
	Notice here the term $16C_{\sigma}C_{\alpha,\EE}=8\sigma^2C_{\alpha,\EE}$ appearing above is dimension $d$ independent. This is because we have used component-wise noises in the system \eqref{sKV}.
	However in \cite[Definition 3.1]{FHPS2} $16C_{\sigma}C_{\alpha,\EE}$ is replaced by $4C_{d,\sigma}C_{\alpha,\EE}=2(d-1)\sigma^2C_{\alpha,\EE}$, which is dimension $d$ dependent due to the isotropic noises used there.
\end{remark}

We shall prove the following result. 
\begin{theorem}\label{thm:mainresult}
	Let us fix $\varepsilon_1>0$ small and assume that the initial datum and parameters $\{\varepsilon_{\alpha^*}, \delta, \theta,\lambda,\sigma\}$ are well-prepared for a time horizon $T^*>0$  and parameter $\alpha^*>0$. Additionally, we assume that  $\rho_0$   has a probability density function (still denoted as $\rho_0$)  being continuous at any global minimizer $v^*$ and $\rho_0(v^*)>0$. Then $E(\rho_{T^*})$  well approximates a minimizer $v^*$ of $\EE$, and the following quantitative estimate holds 
\begin{equation}\label{locest2}
	\left |E(\rho_{T^*})-v^* \right |\leq  \epsilon,
	\end{equation}
	for
\begin{equation}
	\epsilon:=C(C_0,c_1,c_2,\beta)\left((1+C_{\alpha^*,\EE}^{\beta})\left(\frac{\lambda C_{T^*}}{\lambda\theta- 16C_{\sigma}C_{\alpha^*,\EE}}\delta^{\frac{d-2}{4}}+\varepsilon_1\right )^{\min\left \{1,\frac{\beta}{2}\right \}}+\varepsilon_{\alpha^*}^\beta\right)\,,
	\end{equation}
	where $\varepsilon_{\alpha^*}=O(\frac{1}{\alpha^*})$.
\end{theorem}

Next we recall the definition
$
v_{\alpha,\EE}(\rho_t):=\frac{\int_{\SS^{d-1}} v \omega_\alpha^{\EE}(v) d \rho_t(v)}{\|\omega_\alpha^{\EE}\|_{L^{1}(\rho_t)}}=\frac{\int_{\SS^{d-1}} v e^{-\alpha \EE( v )} d \rho_t(v)}{\|e^{-\alpha \EE}\|_{L^{1}(\rho_t)}}\,,
$
and summarize some useful estimates of $v_{\alpha,\EE}(\rho_t)$ and $V(\rho_t)$ from \cite[Lemma 3.1]{FHPS2}. 
\begin{lemma}\label{lemv}
	Let $v_{\alpha,\EE}(\rho_t)$ be defined as above. It holds that
	\begin{enumerate}
		\item 
\begin{align}\label{511}
		\int_{\SS^{d-1}} |v-v_{\alpha,\EE}(\rho_t)|^2d\rho_t\leq 4\frac{e^{-\alpha \underline{\EE}}}{\|\omega_\alpha^{\EE}\|_{L^{1}(\rho_t)}}V(\rho_t)\leq4C_{\alpha,\EE}V(\rho_t)\,;
		\end{align} 
		\item
\begin{align}
		\int_{\SS^{d-1}} |v-v_{\alpha,\EE}(\rho_t)|d\rho_t \leq2\frac{e^{-\alpha \underline{\EE}}}{\|\omega_\EE^\alpha\|_{L^{1}(\rho_t)}}V(\rho_t)^{\frac{1}{2}}\label{eqsi}\leq  2C_{\alpha,\EE}V(\rho_t)^{\frac{1}{2}}\,.
		\end{align}
	\end{enumerate}
\end{lemma}

We also need a lower bound on the norm of the weights $\|\omega_\EE^\alpha \|_{L^1(\rho_t)}$.
\begin{lemma}\label{lemome}
	Let $c_1, c_2$ be the constants from the Assumption \ref{assumas} on $\EE$. Then we have
\begin{equation}\label{eqL1}
	\frac{d}{dt}\| \omega_\EE^\alpha\|_{L^1(\rho_t)}^2 \geq -b_1(\sigma,d,c_1,c_2,\underline{\EE})\alpha^2e^{-2\alpha \underline \EE}V(\rho_t)-b_2(c_1)\lambda\alpha e^{-2\alpha \underline \EE}V(\rho_t)^{\frac{1}{2}} 
	\end{equation}
	with $ b_1,b_2>0$.
\end{lemma}

\begin{proof}
	The derivative of $\|\omega_\EE^\alpha \|_{L^1(\rho_t)}$ is given by
\begin{align*}
	&\frac{d}{dt} \int_{\SS^{d-1}} \omega_\EE^\alpha(v) d\rho_t \notag\\
	=&\int_{\SS^{d-1}} \sum_{i=1}^d\frac{\sigma^2}{2}(\mc{G}_t(\rho))_i^2 \partial_{v_i^S}^2 \omega_\EE^\alpha+ \lambda P(v)v_{\alpha,\EE}(\rho_t)\cdot \nabla_{\SS^{d-1}}  \omega_\EE^\alpha -\frac{\sigma^2}{2}D(\mc{G}_t(\rho))^2v\cdot \nabla_{\SS^{d-1}}  \omega_\EE^\alpha d\rho_t\notag\\
	=&: \textbf{I} + \textbf{II}+\textbf{III}\,,
	\end{align*}
where $\mc{G}_t(\rho)=v-v_{\alpha,\EE}(\rho_t)$.
	The gradient and the Laplacian of the weight function can be computed as
\begin{align}
	\nabla_{\SS^{d-1}} \omega_\EE^\alpha(v) =\nabla \omega_\EE^\alpha\left (\frac{v}{|v|} \right )\bigg|_{|v|=1} =\frac{1}{|v|}\left(I-\frac{vv^T}{|v|^2}\right)\nabla \omega_\EE^\alpha \bigg|_{|v|=1} \,,
	\end{align}
\begin{equation}
	\partial_{v_i^\SS}\omega_\EE^\alpha(v)=\partial_{v_i} \omega_\EE^\alpha-v_iv\cdot \nabla \omega_\EE^\alpha\bigg|_{|v|=1}\,,
	\end{equation}
	and 
\begin{align}
	&\partial_{v_i^\SS}^2\omega_\EE^\alpha(v)= \partial_{v_i}  (\partial_{v_i^\SS}\omega_\EE^\alpha\left (\frac{v}{|v|} \right))\bigg|_{|v|=1}\notag\\
	=&\frac{\partial^2_{v^i}\omega_\EE^\alpha}{|v|}-\frac{\sum_{j=1}^d\partial^2_{v_iv_j}\omega_\EE^\alpha v_jv_i}{|v|^3}-\frac{v\cdot \nabla\omega_\EE^\alpha}{|v|^2}+\frac{v_i^2v\cdot \nabla\omega_\EE^\alpha}{|v|^4}-\frac{v_i \partial_{v_i}\omega_\EE^\alpha}{|v|^2}+\frac{v_i^2v\cdot \nabla\omega_\EE^\alpha}{|v|^4}\notag\\
	&-\frac{\sum_{j=1}^d\partial^2_{v_iv_j}\omega_\EE^\alpha v_jv_i}{|v|^3}+\frac{v_i^2\nabla^2\omega_\EE^\alpha:vv^T}{|v|^5}\bigg|_{|v|=1}\,.
	\end{align}
	We further have
\begin{align*}
	\nabla \omega_\EE^\alpha = -\alpha e^{-\alpha \EE} \nabla \EE ,\quad
	\partial_{v_i} \omega_\EE^\alpha = -\alpha e^{-\alpha \EE}\partial_{v_i} \EE,\quad  \partial^2_{v_iv_j} \omega_\EE^\alpha =\alpha^2 e^{-\alpha \EE}\partial_{v_i} \EE \partial_{v_j} \EE-\alpha e^{-\alpha \EE}\partial^2_{v_iv_j} \EE \,.
	\end{align*}
	We estimate the term $\textbf{I}$ as follows
\begin{align}\label{esI}
	\textbf{I} &= \frac{\sigma^2}{2}\int \sum_{i=1}^d(\mc{G}_t(\rho))_i^2\bigg(\partial^2_{v^i}\omega_\EE^\alpha-2\sum_{j=1}^d\partial^2_{v_iv_j}\omega_\EE^\alpha v_jv_i-v\cdot \nabla\omega_\EE^\alpha\notag\\
	&\qquad+2v_i^2v\cdot \nabla\omega_\EE^\alpha-v_i \partial_{v_i}\omega_\EE^\alpha
	+v_i^2\nabla^2\omega_\EE^\alpha:vv^T\bigg)d\rho_t(v) \notag\\
	&\geq  \frac{\sigma^2}{2}\int \sum_{i=1}^d(\mc{G}_t(\rho))_i^2 \left(-(1+2d+d^2)(\alpha^2c_1^2+\alpha c_2)-4\alpha c_1\right)e^{-\alpha \EE} d \rho_t(v) \notag\\
	& \geq -\frac{\sigma^2}{2}\left((d+1)^2(\alpha^2c_1^2+\alpha c_2)+4\alpha c_1\right) e^{-2\alpha \underline{\EE}}\frac{V(\rho_t)}{\|\omega_{\alpha}^\EE\|_{L^1(\rho_t)}}\,,
	\end{align}
	where we have used that $\|\nabla\EE\|_2\leq c_1$;  $ \|\nabla^2\EE\|_\infty \leq c_2$ and estimate \eqref{511}. For the term $\textbf{II}$ we directly use argument from \cite[Lemma 3.2]{FHPS2} and get
\begin{equation}\label{esII}
	\textbf{II}\geq-4\alpha \lambda c_1e^{-2\alpha \underline \EE}\frac{V(\rho_t)^{\frac{1}{2}}}{\|\omega_{\alpha}^\EE\|_{L^1(\rho_t)}}\,.
	\end{equation}
	For $\textbf{III}$ we compute
\begin{align}\label{esIII}
	\textbf{III}&=\int_{\SS^{d-1}}-\frac{\sigma^2}{2}D(\mc{G}_t(\rho))^2v\cdot \nabla_{\SS^{d-1}}  \omega_\EE^\alpha d\rho_t=\int_{\SS^{d-1}}\alpha\frac{\sigma^2}{2}D(\mc{G}_t(\rho))^2v\cdot (I-vv^T) \nabla\EE e^{-\alpha\EE} d\rho_t\notag\\
	&\geq -\alpha\frac{\sigma^2}{2}c_1e^{-\alpha \underline\EE}\int_{\SS^{d-1}} |\mc{G}_t(\rho)|^2d\rho_t\geq-2c_1\alpha\sigma^2\frac{e^{-2\alpha \underline{\EE}}}{\|\omega_\alpha^{\EE}\|_{L^{1}(\rho_t)}}V(\rho_t)\,,
	\end{align}
	where we have used estimate \eqref{511} in the last inequality.
	
	Combining the inequalities \eqref{esI}, \eqref{esII} and \eqref{esIII} yields
\begin{align}
	\frac{1}{2} \frac{d}{dt}\|\omega_\alpha^\EE \|_{L^1(\rho_t)}^2 & = \|\omega_\alpha^\EE  \|_{L^1(\rho_t)} \frac{d}{dt} \|\omega_\alpha^\EE \|_{L^1(\rho_t)} \notag \\ 
	& \geq -\frac{\sigma^2}{2}\left((d+1)^2(\alpha^2c_1^2+\alpha c_2)+8\alpha c_1\right)e^{-2\alpha \underline \EE}V(\rho_t)-4\alpha \lambda c_1e^{-2\alpha \underline \EE}V(\rho_t)^{\frac{1}{2}}\notag \\
	&=:-b_1(d,\sigma,c_1,c_2,\underline{\EE})\alpha^2e^{-2\alpha \underline \EE}V(\rho_t)-b_2(c_1)\lambda\alpha e^{-2\alpha \underline \EE}V(\rho_t)^{\frac{1}{2}}\,,
	\end{align}
	which completes the proof.
\end{proof}

Next lemma 
provides a well-known quantitative version of Laplace's principle.
\begin{lemma}\label{lemLap} Let $\EE$ fulfill Assumption \ref{assumas} and
	suppose that $\rho_0\in\mc{P}_{ac}(\SS^{d-1})$ has a probability density function (still denoted as $\rho_0$)  on $\SS^{d-1}$ which  is continuous at any global minimizer $v^*$ and $\rho_0(v^*)>0$. Then, we have
\begin{equation*}
	-\frac{1}{\alpha}\log \int_{\SS^{d-1}}e^{-\alpha\EE(v)} d\rho_0(v)=:	-\frac{1}{\alpha}\log \int_{\SS^{d-1}}e^{-\alpha\EE(v)} \rho_0(v)dv=\underline\EE+ 
	\mathcal O \left (\frac{1}{\alpha} \right),\quad \alpha\to\infty\,.
	\end{equation*}
\end{lemma}
\begin{proof}
	Following the proof of \cite[Proposition 3.1]{ha2020convergence}, we choose $D$ to be a ball with radius $2$, which contains the sphere $\SS^{d-1}$.
\end{proof}

Using the above lemmas we can prove the following proposition as in \cite[Proposition 3.1]{FHPS2}.
\begin{proposition}\label{thmE}For any given time horizon $T>0$, assume that 
$$\overline{\mc{V}}_T:=\sup\limits_{0\leq t\leq T}V(\rho_t)\leq \min\left\{\frac{\| \omega_\EE^\alpha\|_{L^1(\rho_0)}^2}{T},\frac{\| \omega_\EE^\alpha\|_{L^1(\rho_0)}^4}{T\lambda^2},\frac{3}{8} \right\}\,.$$
	Then there exists a minimizer $v^*$ of $\EE$ such that  it holds
\begin{equation}
	\left|\frac{E(\rho_t)}{|E(\rho_t)|}- v^\ast\right|\leq C(C_0,c_1,c_2,\beta)\left((C_{\alpha,\EE})^{\beta}V(\rho_t)^{\frac{\beta}{2}}+\varepsilon_\alpha^\beta\right)\quad \mbox{ for all }t\in[0,T]
	\end{equation}
	where $\varepsilon_\alpha=
	O(\frac{1}{\alpha})$, $C_{\alpha,\EE}=e^{\alpha(\overline \EE-\underline \EE)}$, and $C_0$, $c_1$,  $c_2$, $\beta$ are used in Assumption \ref{assumas}. 
\end{proposition}
\begin{proof}
	It follows from Lemma \ref{lemome} that
\begin{align*}
	\| \omega_\EE^\alpha\|_{L^1(\rho_t)}^2 &\geq \| \omega_\EE^\alpha\|_{L^1(\rho_0)}^2-b_1\alpha^2e^{-2\alpha \underline \EE}\int_0^tV(\rho_s)ds-b_2\lambda\alpha e^{-2\alpha \underline \EE}\int_0^tV(\rho_s)^{\frac{1}{2}}ds\\
	&\geq   \| \omega_\EE^\alpha\|_{L^1(\rho_0)}^2-b_1\alpha^2e^{-2\alpha \underline \EE}T\overline{\mc{V}}_T-b_2\alpha e^{-2\alpha \underline \EE}T\lambda\overline{\mc{V}}_T^{\frac{1}{2}}\\
	&\geq   \| \omega_\EE^\alpha\|_{L^1(\rho_0)}^2-b_1\alpha^2e^{-2\alpha \underline \EE}\| \omega_\EE^\alpha\|_{L^1(\rho_0)}^2-b_2\alpha e^{-2\alpha \underline \EE}\| \omega_\EE^\alpha\|_{L^1(\rho_0)}^2\,,
	\end{align*}
	where we have used the assumption  
$
	\overline{\mc{V}}_T\leq \min \left\{\frac{\| \omega_\EE^\alpha\|_{L^1(\rho_0)}^2}{T},\frac{\| \omega_\EE^\alpha\|_{L^1(\rho_0)}^4}{T\lambda^2}\right\}.
$
	The above inequality implies
\begin{align*}
	-\frac{1}{\alpha}\log \|\omega_\EE^{\alpha}\|_{L^1(\rho_t)}\leq -\frac{1}{\alpha}\log \|\omega_\EE^{\alpha}\|_{L^1(\rho_0)}-\frac{1}{2\alpha}\log\left(1-b_1\alpha^2e^{-2\alpha \underline \EE}-b_2\alpha e^{-2\alpha \underline \EE}\right)\,.
	\end{align*}
	Note that Lemma \ref{lemLap} states
\begin{equation}
	-\frac{1}{\alpha} \log \|\omega_\EE^{\alpha}\|_{L^1(\rho_0)}-\underline\EE=
	\mathcal O \left (\frac{1}{\alpha} \right),\quad \alpha\to\infty\,,
	\end{equation}
	which yields that
\begin{align*}
	&-\frac{1}{\alpha}\log \|\omega_\EE^{\alpha}\|_{L^1(\rho_t)} -\underline \EE \leq -\frac{1}{\alpha}\log \|\omega_\EE^{\alpha}\|_{L^1(\rho_0)}-\underline{\EE}-\frac{1}{2\alpha}\log\left(1-b_1\alpha^2e^{-2\alpha \underline \EE}-b_2\alpha e^{-2\alpha \underline \EE}\right)\\
	&=
	\mathcal O \left (\frac{1}{\alpha} \right)-\frac{1}{2\alpha}\log\left(1-b_1\alpha^2e^{-2\alpha \underline \EE}-b_2\alpha e^{-2\alpha \underline \EE}\right)\leq 
	\mathcal O \left (\frac{1}{\alpha} \right)=:\varepsilon_{\alpha}
	\,.
	\end{align*}
	Let us assume that $\overline{\mc{V}}_T\leq \frac{3}{8}$, then $\frac{1}{2}\leq |E(\rho_t)|\leq 1\,.$ Following the same argument as in \cite[Proposition 3.1]{FHPS2}, we obtain
\begin{align*}
	\left|-\frac{1}{\alpha}\log \|\omega_\EE^{\alpha}\|_{L^1(\rho_t)} -\EE\left({\frac{E(\rho_t)}{|E(\rho_t)|}}\right )\right|\leq  2\sqrt{\frac{2}{3}} c_1C_{\alpha,\EE}V(\rho_t)^{\frac{1}{2}}\,.
	\end{align*}
	Hence we have
\begin{align*}
	0\leq \EE\left ({\frac{E(\rho_t)}{|E(\rho_t)|}} \right )-\underline \EE &\leq  \EE\left({\frac{E(\rho_t)}{|E(\rho_t)|}} \right)-\frac{-1}{\alpha}\log \|\omega_\EE^{\alpha}\|_{L^1(\rho_t)} +\frac{-1}{\alpha}\log \|\omega_\EE^{\alpha}\|_{L^1(\rho_t)} -\underline \EE\\
	&\leq   2\sqrt{\frac{2}{3}}c_1C_{\alpha,\EE}V(\rho_t)^{\frac{1}{2}}+\varepsilon_\alpha\,,
	\end{align*}
	which yields that
\begin{align*}
	\left|\frac{E(\rho_t)}{|E(\rho_t)|}- v^*\right|\leq C_0\left|\EE\left(\frac{E(\rho_t)}{|E(\rho_t)|} \right)-\underline \EE\right|^\beta \leq  C(C_0,c_1,\beta)\left((C_{\alpha,\EE})^{\beta}V(\rho_t)^{\frac{\beta}{2}}+\varepsilon_\alpha^\beta\right)\,.
	\end{align*}
	by the inverse continuity $6.$ in Assumption \ref{assumas}, where $v^*$ is a minimizer of $\EE$.	
\end{proof}

The next ingredient is proving the monotone decay of the variance $V(\rho_t)$ under assumptions of well-preparation (see Definition \ref{def:wellprep}).

\begin{proposition}\label{mainp}Let us fix $T>0$ and choose $\alpha$ such that the parameters and the initial datum are well-prepared in the sense of Definition \ref{def:wellprep}. Then it holds
\begin{equation}
	V(\rho_t)\leq V(\rho_0)e^{-(\lambda\theta- 4C_{\alpha,\EE}C_{\sigma,d})t}+\frac{\lambda C_T}{\lambda\theta- 16C_{\sigma}C_{\alpha,\EE}}\delta^{\frac{d-2}{4}}\quad\mbox{ for all }t\in[0,T] \,.
	\end{equation}
\end{proposition}
\begin{proof} 
	Let us compute the derivative of the variance (where $C_{\sigma}=\frac{\sigma^2}{2}$)
\begin{align*}
	&\frac{d}{dt}V(\rho_t) = \frac{1}{2}\frac{d}{dt}\bigg( \int_{\SS^{d-1}} v^{2}d\rho_t-E(\rho_t)^{2}\bigg)=\frac{1}{2}\frac{d}{dt}\bigg( 1-E(\rho_t)^{2}\bigg)= - E(\rho_t)\frac{d}{dt}E(\rho_t)\\
	&= E(\rho_t)\int_{\SS^{d-1}}\eta_t d\rho_t + C_{\sigma}\int_{\SS^{d-1}}(\mc{G}_t(\rho))^{2}\la E(\rho_t),v \ra d\rho_t \\
	&\quad +C_{\sigma}\int_{\SS^{d-1}}\la E(\rho_t), D(\mc{G}_t(\rho))^2v \ra d\rho_t+2C_{\sigma}\int_{\SS^{d-1}} |D(\mc{G}_t(\rho))v|^2 \la E(\rho_t),v \ra d\rho_t\\
	&= \lambda \int_{\SS^{d-1}}\langle v_{\alpha, \EE}, v \rangle  \la E(\rho_t), v\ra -\la E(\rho_t), v_{\alpha, \EE}\ra d\rho_t + \mathbf{D}\,,
	\end{align*}
	where $\mathbf{D}$ denotes the diffusion term
\begin{align*}
	\mathbf{D}&:=C_{\sigma}\int_{\SS^{d-1}}(\mc{G}_t(\rho))^{2}\la E(\rho_t),v \ra d\rho_t+C_{\sigma}\int_{\SS^{d-1}}\la E(\rho_t), D(\mc{G}_t(\rho))^2v \ra d\rho_t \\
	&\quad +2C_{\sigma}\int_{\SS^{d-1}} |D(\mc{G}_t(\rho))v|^2 \la E(\rho_t),v \ra d\rho_t\,.
	\end{align*}
	Note that comparing to \cite[Proposition 3.2]{FHPS2}, the only difference here is the diffusion term $\textbf{D}$.
	Applying estimate \eqref{511} it is easy to obtain
	that
\begin{equation}
	\mathbf{D}\leq 4C_{\sigma}\int_{\SS^{d-1}}|\mc{G}_t(\rho)|^{2}d\rho_t\leq 16C_{\sigma}C_{\alpha,\EE}V(\rho_t)\,.
	\end{equation}
	
	By the  assumption that $\rho_0\in L^2(\SS^{d-1})$ (see Definition \ref{def:wellprep}), we have the solution $\rho_t$ is not just a measure but it is a square integrable function, and for any given $T>0$ it satisfies $\rho\in L^\infty([0,T];L^2(\SS^{d-1}))$. This can be proved through a standard argument of PDE theory, for which we refer to \cite[Theorem 4.1]{FHPS2} or \cite[Theorem 2.4]{albi2017mean}. The rest of proof  is precisely the same as the proof in \cite[Proposition 3.2]{FHPS2}, where we only need to replace $4C_{\alpha,\EE}C_{\sigma,d}$ by $16C_{\sigma}C_{\alpha,\EE}$.
\end{proof}

\begin{proof} (of {\bf Theorem \ref{thm:mainresult}})	
	The proof follows the same arguments as in \cite[Theorem 3.1]{FHPS2} by using Proposition \ref{mainp} and Proposition \ref{thmE}.
\end{proof}

\section{Numerical implementation and tests}\label{sec:num}

In this section we present several tests and examples of application of the CBO method based on the anisotropic stochastic Kuramoto-Vicsek (KV) system. First, we briefly discuss some implementation aspects, including accelerated algorithms and convergence criteria (see also \cite{FHPS2} for more details). Next, we test the method against its corresponding isotropic version \cite{FHPS1, FHPS2} with respect to some well-known prototype test functions in high dimensions. We consider a wide range of test function, as well as, machine learning applications like robust linear regression and the phase retrieval problem.
\begin{figure}[H]
	\centering{\small
		\includegraphics[scale=0.225]{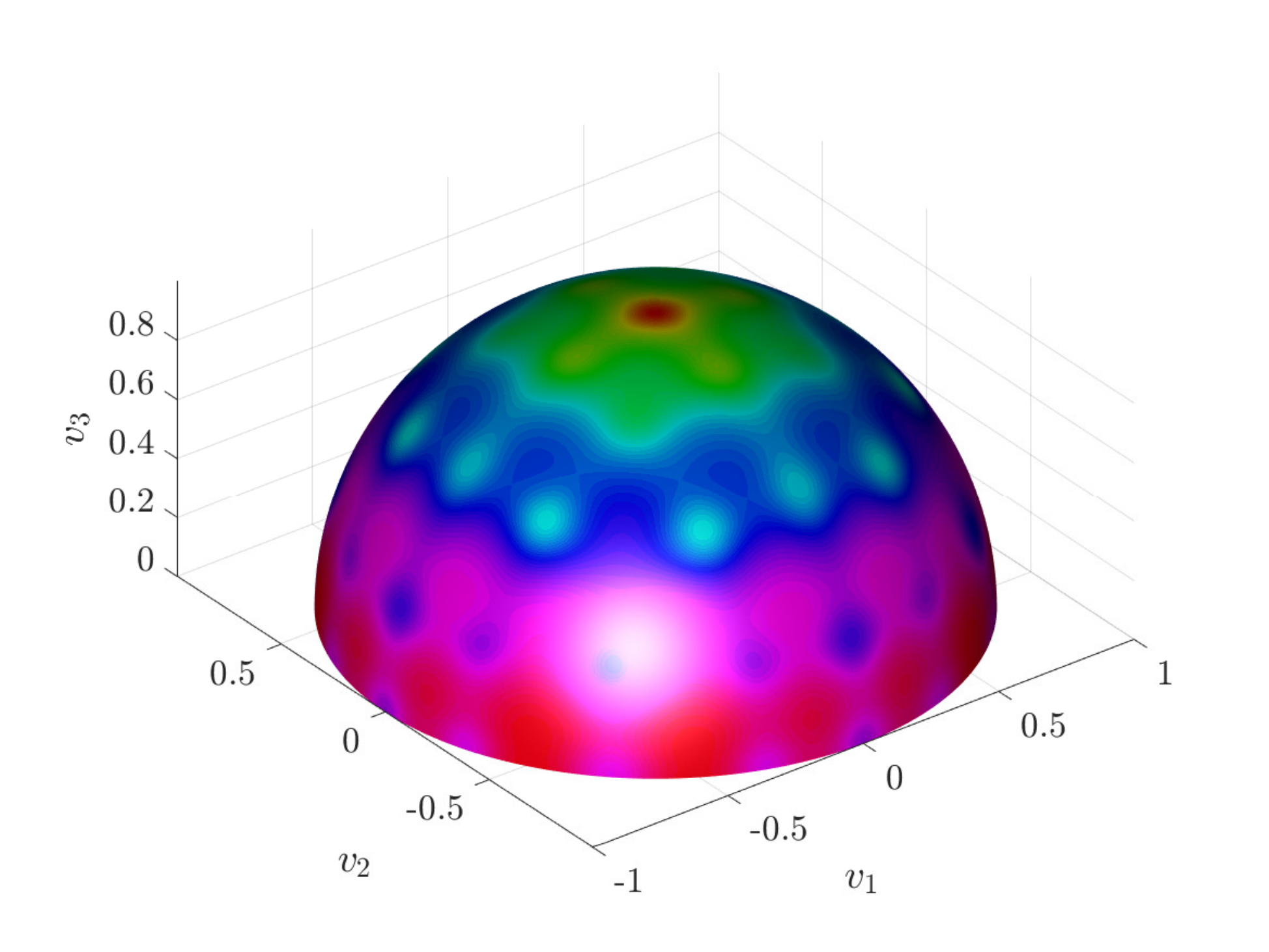}\hskip -.5cm
		\includegraphics[scale=0.225]{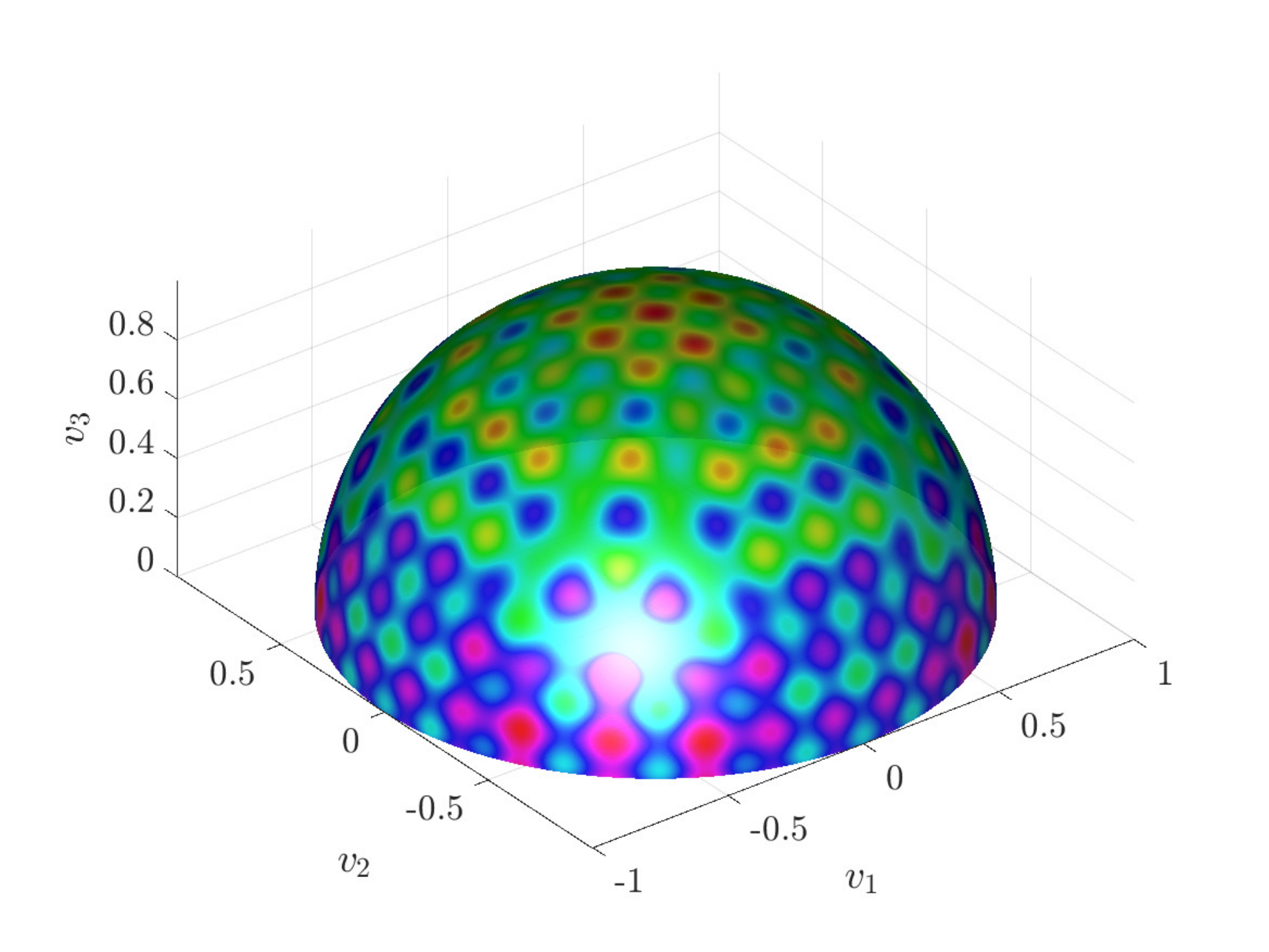}\hskip -.5cm
		\includegraphics[scale=0.225]{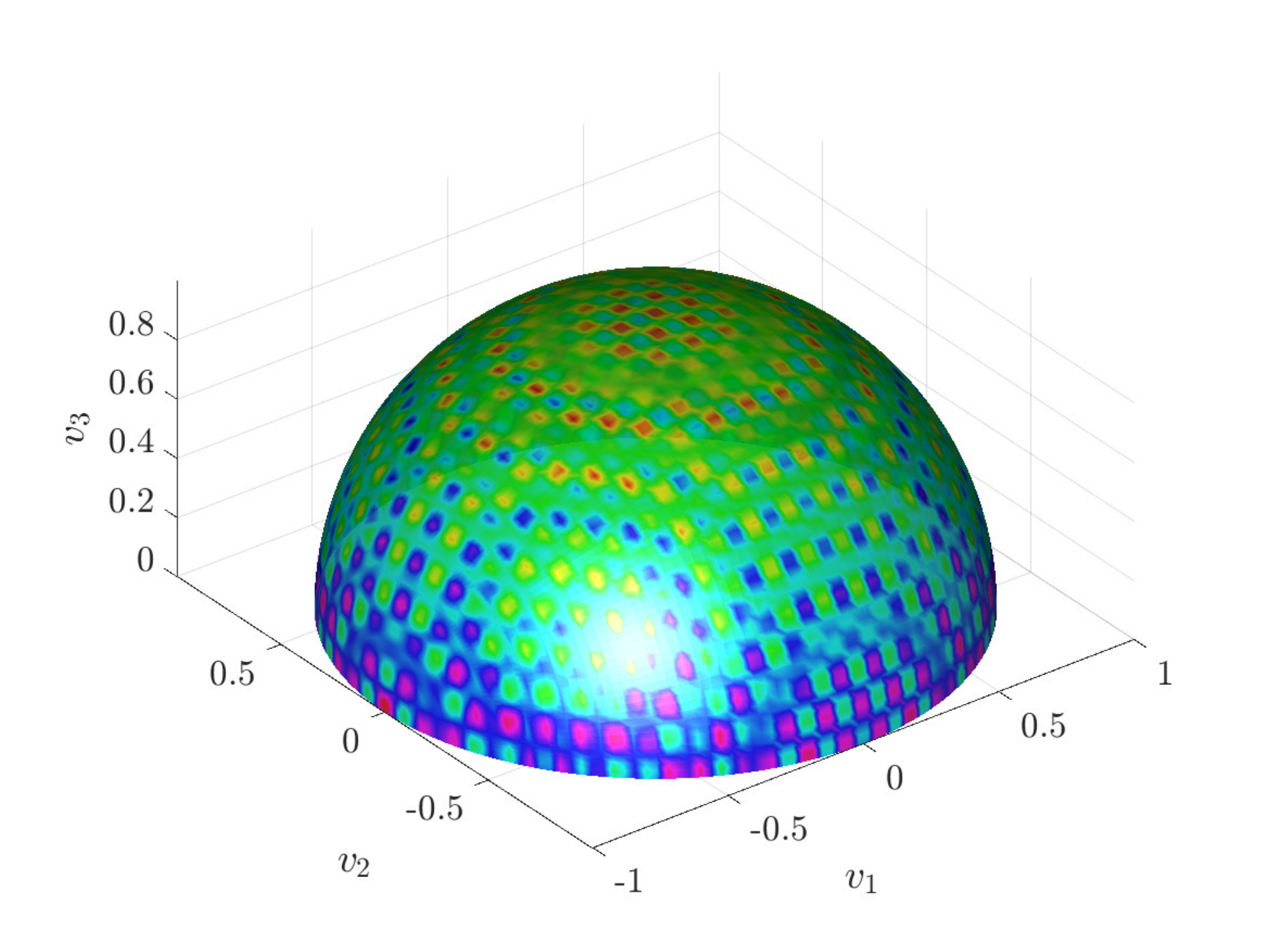}\\[-.25cm]
		\hskip 1cm
		(a) Ackley \hskip 2cm (b) Rastrigin \hskip 2cm (c) Griewank \\[-.05cm]
		\includegraphics[scale=0.225]{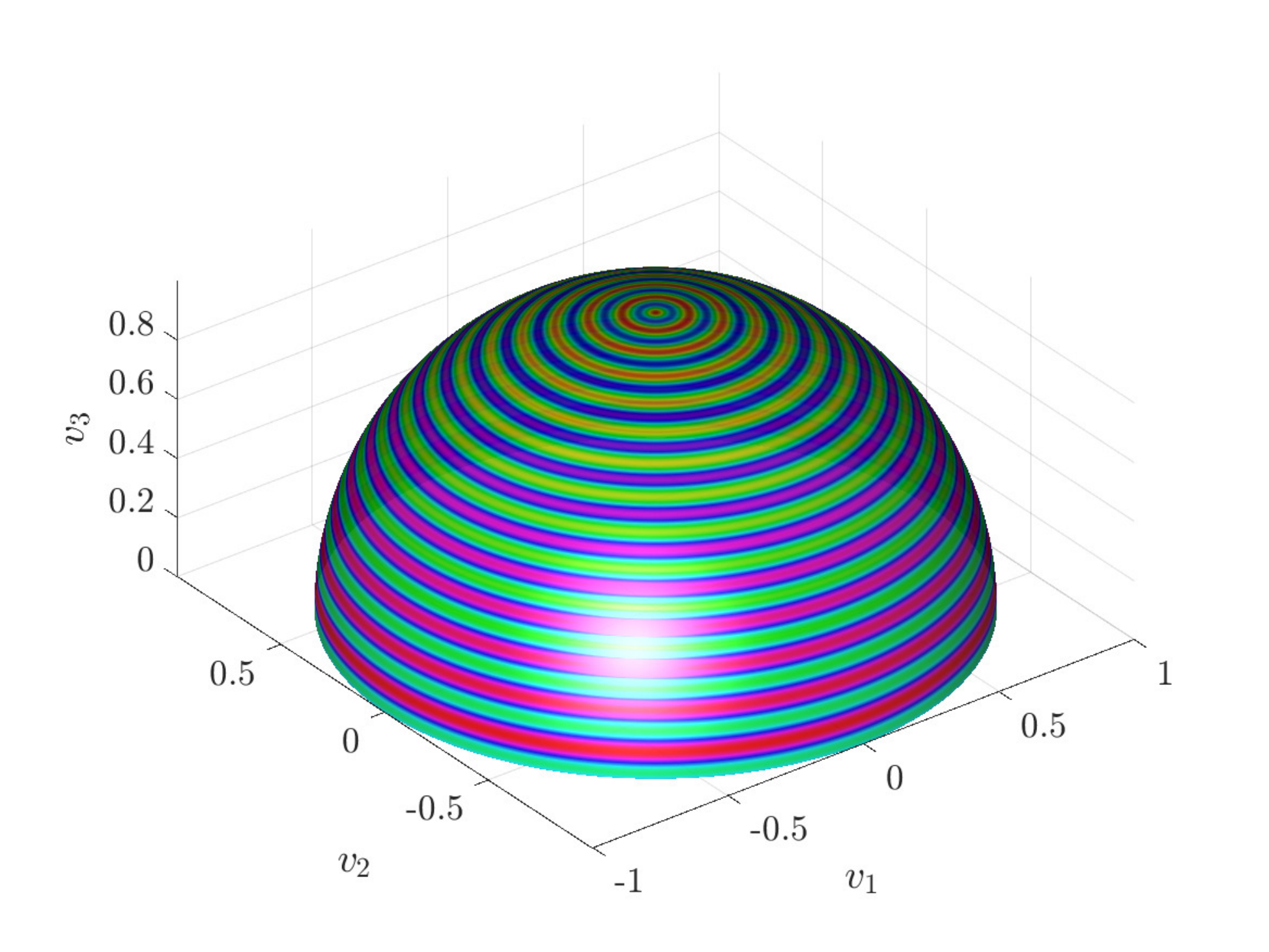}\hskip -.5cm
		\includegraphics[scale=0.225]{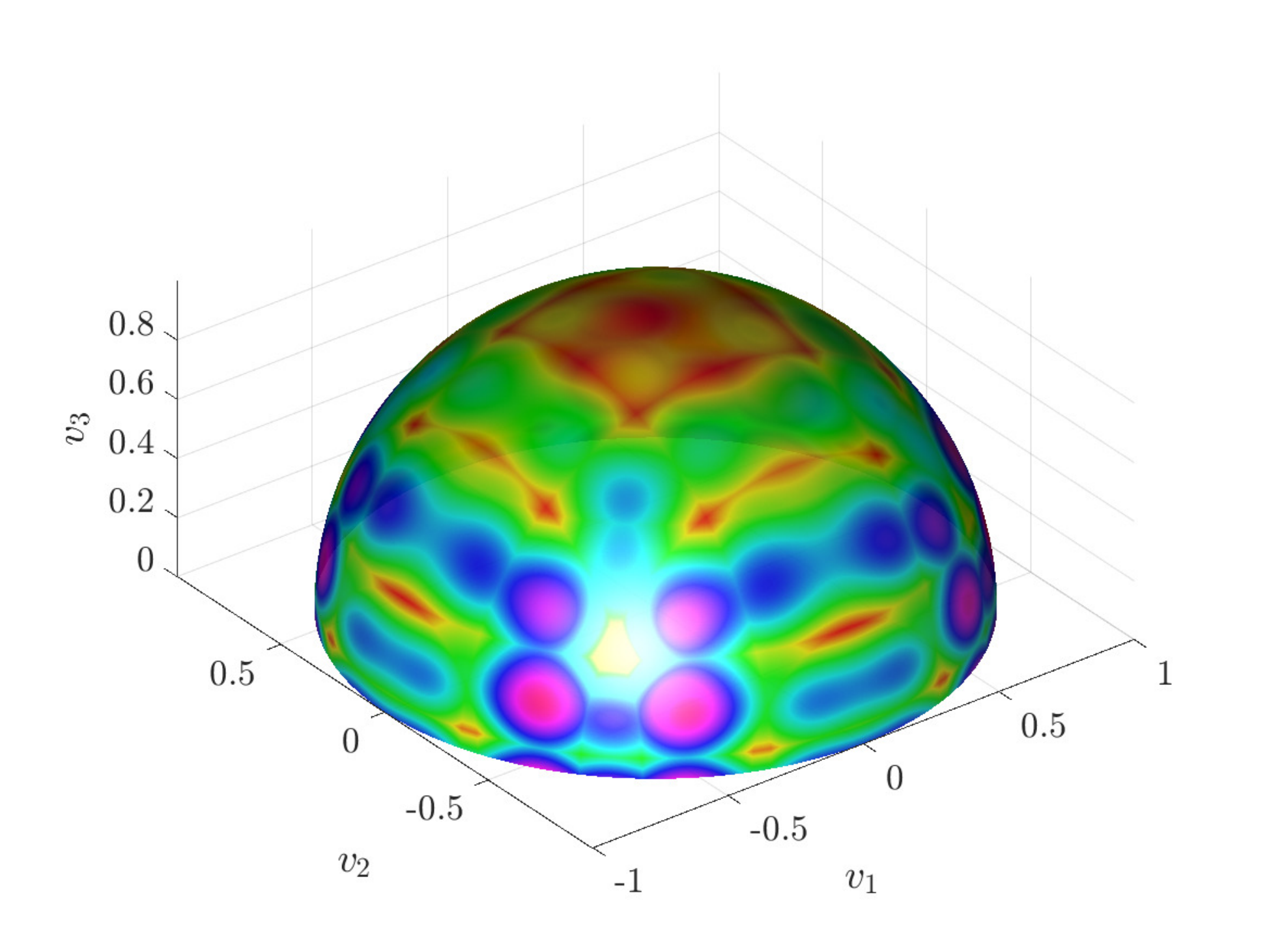}\hskip -.5cm
		\includegraphics[scale=0.225]{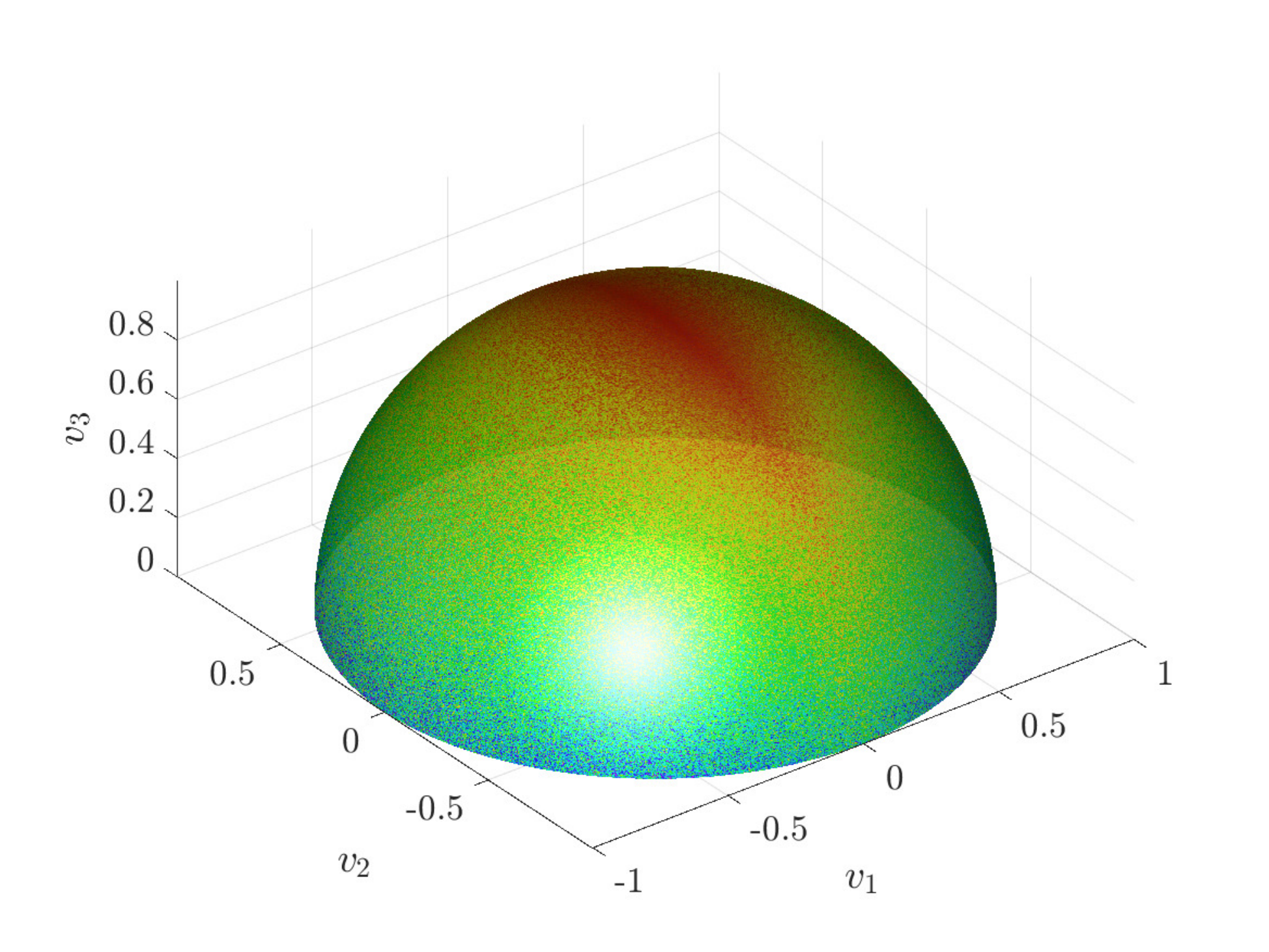}\\[-.25cm]
		\hskip 1cm
		(d) Salomon  \hskip 2cm (e) Alpine \hskip 2cm (f) X.-S.-Y. random}
	\caption{Test functions in dimension $d=3$ constrained over the sphere $\SS^2$. Functions are plotted as contour lines over the surface of the sphere.
		For visualization purposes only the upper part of the sphere is reported and in some cases the size of the search space has been reduced. The global minimum corresponds to the direction $v^*=(0,0,1)^T$ in all cases.}
	\label{fg:ackley}
\end{figure}

\subsection{Discretization of the anisotropic KV system}

We discuss the discretization of the KV system \eqref{sKV}. Since we consider constrained dynamics on a hypersphere, we rely on a projection scheme of the general form
\begin{equation}
\left\lbrace
\begin{aligned}
\widetilde V^i_{n+1}&=V^i_{n}+\Phi(\Delta t, V^i_n,\Delta B_n^i),\\
V^i_{n+1}&=\frac{\widetilde V^i_{n+1}}{|\widetilde V^i_{n+1}|},
\label{eq:gen2}
\end{aligned}
\right.
\end{equation}
where the function $\Phi(\Delta t,\cdot,\Delta B_n^i):\RR^{d}\to\RR^{d}$ defines the scheme, $\Delta t$ is the time step, $V_n^i \approx V^i_t|_{t=n \Delta t}$ is the $i$th agent at time $n\Delta t$, and $\Delta B_n^i=B^i_{{n+1}\Delta t}-B^i_{{n}\Delta t}$ are independent normal random vectors sampled from $N(0,\Delta t)$. \\
As efficiency of the numerical solver in high-dimension is of paramount importance, in our numerical experiments we rely on projection methods of the type \eqref{eq:gen2} based on the simple Euler-Maruyama scheme
\begin{align}
\Phi(\Delta t, V^i_n,\Delta B_n^i) =\ & \Delta t\lambda P(V_n^i)V_n^{\alpha,\EE} + \sigma  P(V_n^i)D_{n,i}\Delta B_n^i \label{eq:EM2} \\
&\ -\Delta t \frac{\sigma^2}{2}\left(|V_n^i-V_n^{\alpha,\EE}|^2+D_{n,i}^2-2\left|D_{n,i}V_n^i\right|^2\right)V^i_n\quad i=1,\cdots,N\,.\notag
\end{align}
In the sequel we analyze in more details some computational aspects and improvements related to the standard approach based on a direct application of \eqref{eq:gen2}-\eqref{eq:EM2}. Let us point out that the set of three computational parameters, $\Delta t$, $\sigma$ and $\lambda$, defining the scheme can be reduced, since we can rescale the time by setting $\tau=\lambda \Delta t$, $\nu^2 = \frac{\sigma^2}{\lambda}$, 
to obtain a scheme, which depends only on two parameters $\tau$ and $\nu$. 

\subsubsection*{Evaluation of $V_n^{\alpha,\EE}$}

Let us observe that the computation of $V_n^{\alpha, \EE}$ is crucial and that a straightforward evaluation using $V_n^{\alpha, \EE} = \frac1{N_\alpha} \sum_{j=1}^N w_\alpha^{\EE}(V^j_n)V^j_n$ for $N_\alpha = \sum_{j=1}^N w_\alpha^{\EE}(V^j_n)$, where $w_\alpha^{\EE}(V^j_n)=\exp(-\alpha\EE(V^j_n))$ is generally unstable since for large values of $\alpha \gg 1$ the value of $N_\alpha$ is close to zero. On the other hand, the use of large values of $\alpha$ is essential for the performance of the method. A way to overcome this  is based on the following numerical technique
\begin{equation*}
\frac{w_\alpha^{\EE}(V^j_n)}{N_{\alpha}} = \frac{\exp(-\alpha\EE(V^j_n))}{\sum_{j=1}^N \exp(-\alpha\EE(V^j_n))}\frac{\exp(\alpha\EE(V_n^*))}{\exp(\alpha\EE(V_n^*))} =\frac{\exp(-\alpha(\EE(V_n^j)-\EE(V_n^*)))}{\sum_{j=1}^N \exp (-\alpha(\EE(V_n^j)-\EE(V_n^*)))}
\end{equation*}
where $V_n^{*} := \argmin_{V\in\{V_n^i\}^N_{i=1}} \EE(V)$,
is the location of the agent with the minimal function value in the current population.
This ensures that for at least one agent $V^j_n=V_n^{*}$ , we
have $\EE(V_n^j)-\EE(V_n^*) = 0$ and therefore, $\exp(-\alpha(\EE(V_n^j)-\EE(V_n^*)))=1$. For the sum this
leads to $\sum_{j=1}^N \exp (-\alpha(\EE(V_n^j)-\EE(V_n^*)))\geq 1$, so that the division does not induce a numerical problem.  

\subsubsection*{Batch algorithms}

The computation of $V_n^{\alpha, \EE}$ may be accelerated by using the random approach presented in \cite{AlPa} (see Algorithm 4.7). Namely, by considering a random subset $J_M$ of size $M < N$ of the indexes $\{1,\ldots,N\}$ and computing $V_n^{\alpha,\EE,J_M} = \frac1{N^{J_M}_\alpha} \sum_{j\in J_M} w_\alpha^{\EE}(V^j_n)V^j_n$,  $N^{J_M}_\alpha = \sum_{j\in J_M} w_\alpha^{\EE}(V^j_n)$. Similarly, we will stabilize the above computation by centering it at $V_n^{J_M,*} := \argmin_{V\in\{V_n^j\}_{j\in J_M}} \EE(V)$.
The random subset is typically chosen at each time step. As a further randomization variant, at each time step, we may partition agents into disjoint subsets $J^k_M$, $k=1,\ldots,S$ of size $M$ such that $SM=N$ and compute the evolution of each batch separately (see \cite{JLJ, carrillo2019consensus} for more details). 

\subsubsection*{Convergence criteria}
There are different convergence criteria that can be adopted. In the following experiments, besides checking convergence to a minimizer, we adopt the following standard criteria in heuristic global minimization algorithms. We check that the absolute change in the value of $V_n^{\alpha, \EE}$ over the last $n_{stall}$ iterations is less than a given tolerance $\delta_{stall}$. More precisely, we stop the iteration if $|V^{\alpha, \EE}_{n}-V_{n-1}^{\alpha, \EE}| < \delta_{stall}$,
for $n_{stall}$ consecutive iterations or the maximum number of iterations $n_{T}$ has been reached. 

\subsubsection*{Fast algorithms}

Since we expect that asymptotically the variance of the system goes to zero because of the consensus dynamics, we may accelerate the simulation by discarding agents in time accordingly to the variance of the system \cite{AlPa}. This also influences the computation of $V_n^{\alpha,\EE}$ by increasing the randomness and reducing the possibilities to get trapped in a local minimum. For a set of $N_{n}$ agents, let us define at the time $(n+1)\Delta t$ the empirical variance as $\Sigma_{n+1} = \frac1{N_{n}}\sum_{j=1}^{N_n} (V_{n+1}^j-\bar{V}_{n+1})^2$, $\bar{V}_{n+1} = \frac1{N_n}\sum_{j=1}^{N_n} V^j_{n+1}$.
In the case where the trend to consensus is monotonic $\Sigma_{n+1} \leq \Sigma_n$, we can discard agents uniformly at the time step $(n+1)\Delta t$  accordingly to the ratio $\Sigma_{n+1}/\Sigma_n \leq 1$ without affecting their theoretical distribution. One way to realize this is to define the new number of agents as  $N_{n+1}=\left[\!\!\left[N_n \left(1+\mu\left(\frac{\Sigma_{n+1}-\Sigma_{n}}{\Sigma_{n}}\right)\right)\right]\!\!\right]$,
where $[\![\,\cdot\,]\!]$ denotes the integer part and $\mu\in [0,1]$. For $\mu=0$ we have the usual algorithm where no agents are discarded whereas for $\mu=1$ we achieve the maximum speed up.

We report in Algorithm \ref{algo:KV-CBOfc} the details of the method, which includes the speed-up techniques just discussed. 

 \begin{algorithm}
\caption{Fast KV-CBO method}
\label{algo:KV-CBOfc}
{\small \begin{algorithmic}
\STATE{Set $N_0 = N$ and generate $V_0^i$, $i=1,\ldots,N_0$ sample vectors uniformly on $\SS^{d-1}$}
\STATE{Compute the variance $\Sigma_0$ of $V_0^i$}
\FOR{$n=0$  {\bf to} $n_{T}$}
\STATE{Generate $\Delta B_n^i$ independent normal random vectors $N(0,\Delta t)$}
\IF{$M \leq N_n$}
\STATE{select a batch $J_M$ and compute $V_n^{\alpha, \EE}$ }
\ELSE
\STATE{use \eqref{eq:valpha}}
\ENDIF
\STATE{$\tilde V^i_{n+1}\gets V^i_n +\Phi(\Delta t, V^i_n,\Delta B_n^i)$}
\STATE{$V^i_{n+1}\gets \tilde V^i_{n+1}/|\tilde V^i_{n+1}|$, $i=1,\ldots,N_n$}
\STATE{Compute the variance $\Sigma_{n+1}$ of $V_{n+1}^i$}
\STATE{Set $N_{n+1}\gets \max\{ N_{min}, [\![N_n \left(1+\mu\left((\Sigma_{n+1}-\Sigma_{n})/\Sigma_{n}\right)\right)]\!] \}$ and discard uniformly $N_{n}-N_{n+1}$ samples}
\ENDFOR
\end{algorithmic}}
\end{algorithm}

\subsection{Numerical experiments}
\subsubsection{Test functions constrained over a $d$-dimensional sphere}\label{sec:numsynt}

In this subsection we consider some classical non convex test functions \cite{JYZ} constrained over $\mathbb{S}^d$. All functions have global minimum at $\Vb=(0,0,\ldots,1)^T$ (see Figure \ref{fg:ackley} for $d=3$). \\
(a) The Ackley function:
\begin{equation}
\EE(V)= -A \exp\left(-\frac{ab}{\sqrt{d}}|V-\Vb|\right)-\exp\left(\frac1{d}\sum_{k=1}^{d} \cos(2\pi b(V_k-\Vb_k))\right)+e+B,
\end{equation}
with $A=20$, $a=0.2$, $b=32$, $B=20$. The Ackley function has several local minima in a nearly flat outer region, and a large hole at the centre. \\
(b) The Rastrigin function:
\begin{equation}
\EE(V)= \frac{b^2}{d}|V-\Vb|^2-\frac{A}{d}\sum_{k=1}^{d}\cos(2\pi b (V_k-\Vb_k)) + B,
\end{equation}
with $A=10$, $b=5.12$ and $B=10$. The Rastrigin function has many widespread local minima. It is highly multi-modal with locations of the minima regularly distributed.  \\
(c) The Griewank function:
\begin{equation}
\EE(V)= Ab^2|V-\Vb|^2-\prod_{k=1}^d \cos\left(\frac{b(V_k-\Vb_k)}{\sqrt{k}}\right)+B,
\end{equation}
with $A=1/4000$, $b=600$, $B=1$. The function is non separable and has a huge number of regularly distributed local minima in the search space.\\
(d) The Salomon function:
\begin{equation}
\EE(V)= A\cos\left(2\pi b|V-\Vb| \right) + ab|V-\Vb|+B,
\end{equation}
with $a=0.1$, $b=100$, $A=-1$, $B=1$. It is non-separable and has a large number of local minima. The global minimum has a small area relative to the search space.  \\
(e) The Alpine function:
\begin{equation}
\EE(V)=b\sum_{k=1}^d \left|(V_k-\Vb_k)\sin(b(V_k-\Vb_k))-a(V_k-\Vb_k)\right|,
\end{equation}
with $a=0.1$ and $b=10$. The function has several local minima and is non-differentiable.\\
(f) The Xin-She Yang (XSY) stochastic function:
\begin{equation}
\EE(V)= \sum_{k=1}^{d}\xi_k|b(V_k-\Vb_k)|^k,
\end{equation}
with $b=5$ and $\xi_k$ uniform random variables in $[0,1]$. The function takes random values and is non-differentiable.
\\
In all our simulations we initialize the agents with a uniform distribution over the sphere \cite{Marsaglia, Muller}. We count one run as successful if  $\|V_{n_T}^{\alpha,\EE}-v^*\|_{\infty} \leq 0.05$,
where $V_{n_T}^{\alpha,\EE}$ is the minimizer found by the KV method at the final time $n_T$. We also compute the expected error in the computation of the minimum by considering averages of $|V_{n_T}^{\alpha,\EE}-v^*|$ over $100$ runs. In all computations we have fixed $\delta_{stall}=10^{-4}$, $n_{stall}=250$ and $n_{T}=20000$. We compare the results obtained using the isotropic KV method \cite{FHPS1, FHPS2} with the anisotropic KV method proposed in this paper. Here we did not try to compute the optimal set of parameters for each test case, but for a given dimension $d$ we fix for each scheme a value of $\Delta t$ and $\sigma$ and consider various possible values of $N$ and $M$. The values $\alpha=5\times 10^4$ and $\mu=0.1$ have been selected for both solvers in all examples as a good compromise between efficiency and accuracy. The minimum number of agents $N_{min}$ has been fixed to $10$ and the variance reduction test has been performed each $10$ iterations. In the following table we  reported the rate of success, the final error $\|V_{n_T}^{\alpha,\EE}-v^*\|_{\infty}$, the average number of agents during the simulation $N_{avg}$ and the average number of iterations $n_{avg}$ needed. As a consequence, a measure of the computational cost of the simulation is obtained as $N_{avg} \times n_{avg}$.

In table \ref{tb:2} we report the results for $d=20$ using a variable number of agents $N$ between $50$ and $200$. The batch size was chosen as the $60\%$ of the initial number of agents. The specific values and the corresponding batch sizes $M$ are shown in the table. The value of $\sigma$ in the isotropic case has been taken in order to match the condition $\sigma^2 (d-1) < 2$. In this high dimensional case it is clear that the isotropic KV-CBO method has difficulties when functions have a strong multi-modal behavior like the Rastrigin, Alpine and XYS random functions. For the Rastrigin and XSY random functions, the success rates of the isotropic KV were always at $0\%$. For the Alpine functions the isotropic KV yielded only slightly better success rates of $5\%$ if we choose a larger number of agents. On the other hand, the anisotropic KV-CBO method reaches successful rates between 85\% and 100\% for these three functions. The convergence to the minimum in the case of the Salomon function is extremely slow for the isotropic method that reaches the maximum number of iterations allowed. The only exception is the case of the Griewank function, where the isotropic method has proven to be more efficient and more accurate.
We report in Table \ref{tb:2b} the results for the Rastrigin and XSY random function for a specific set of parameters which permits to recover 100\% success rate with the anisotropic method. Finally 
 we also considered the minimum rotated by an angle $\pi/8$ from a cardinal point. This test is extremely challenging for both methods, and for the same set of parameters optimal convergence properties are observed for Ackley, Griewank, and Salomon functions, while no convergence is found for Rastrigin and XYS random. 
However,  both KV-methods can solve perfectly these latter problems as well, if the initial data is sufficiently concentrated around the global minimum, e.g., according to a von-Mises-Fisher distribution.

\begin{table}[htb]
\caption{Comparison of isotropic ($\sigma=0.3$, $\Delta t=0.05$) and anisotropic ($\sigma=5$, $\Delta t=0.0025$) KV-CBO methods for various agent numbers for the selected test functions constrained on $\mathbb{S}^d$ with $d=20$.}
\begin{center}
{\footnotesize
\begin{tabular}{l|lccc|ccc}
\multicolumn{1}{l}{\rm Function} & &            \multicolumn{3}{c}{Isotropic KV-CBO} & \multicolumn{3}{c}{Anisotropic KV-CBO} \\
\hline
\hline
 & & $N=50$ & $N=100$ & $N=200$ & $N=50$ & $N=100$ & $N=200$\\
               & & $M=30$ & $M=60$ & $M=120$& $M=30$ & $M=60$ & $M=120$\\
\hline
{\rm Ackley} & Rate & 100\% & 100\% & 100\% & 100\% & 100\% & 100\%\\ 
& Error & 3.61e-02 & 3.30e-02 & 2.49e-02 & 1.31e-02 & 2.99e-03 & 7.52e-04\\
& $N_{avg}$ & 23.2 & 30.4 & 53.3 & 24.4 & 41.4 & 79.7\\
& $n_{avg}$ & 2610.7 & 2040.2 & 1821.6 & 2803.1 & 2561.4 & 2365.3\\
                     
\hline
{\rm Rastrigin} & Rate & 0\% & 0\% & 0\% & 73\% & 83\% & 92\%\\ 
& Error & - & - & - & 7.57e-03 & 2.68e-03 & 1.53e-03\\
& $N_{avg}$ & 24.5 & 28.2 & 45.5 & 21.6 & 40.6 & 70.3\\
& $n_{avg}$ & 2884.3 & 1990.9 & 1739.4 & 2989.1 & 2540.3 & 2077.8\\
\hline
{\rm Griewank} & Rate & 100\% & 100\% & 100\% & 100\% & 100\% & 100\%\\ 
& Error & 2.02e-02 & 1.95e-02 & 2.14e-02 & 2.01e-02 & 2.23e-02 & 2.46e-02\\
& $N_{avg}$ & 26.7 & 37.2 & 66.0 & 25.0 & 44.2 & 84.9\\
& $n_{avg}$ & 2535.9 & 2088.2 & 1787.7 & 2893.2 & 2637.8 & 2504.8\\
        \hline
{\rm Salomon} & Rate & 100\% & 100\% & 100\% & 100\% & 100\% & 100\%\\ 
& Error & 1.89e-02 & 1.85e-02 & 1.62e-02 & 3.76e-02 & 2.38e-02 & 1.85e-02\\
& $N_{avg}$ & 13.9 & 16.7 & 25.7 & 13.8 & 18.1 & 28.2\\
& $n_{avg}$ & 20000 & 20000 & 20000 & 5069.8 & 5366.1 & 5746.8\\        
        \hline
{\rm Alpine} & Rate & 0\% & 2\% & 5\% & 94\% & 99\% & 100\%\\ 
& Error & - & 2.65e-02 & 3.10e-02 & 2.65e-02 & 2.74e-02 & 2.66e-02\\
& $N_{avg}$ & 13.9 & 15.7 & 22.8 & 14.2 & 17.2 & 24.2\\
& $n_{avg}$ & 5442.7 & 4606.9 & 4110.9 & 2341.2 & 2126.3 & 1991.0\\      
        \hline
{\rm XSY random} & Rate & 0\% & 0\% & 0\% & 60\% & 78\% & 85\%\\ 
& Error & - & - & - & 7.25e-02 & 7.28e-02 & 6.46e-02\\
& $N_{avg}$ & 16.7 & 20.5 & 33.3 & 14.5 & 18.7 & 27.2\\
& $n_{avg}$ & 20000.0 & 20000.0 & 20000.0 & 8248.5 & 7370.4 & 6549.0\\       
\hline
\hline
\end{tabular}
}
\end{center}
\label{tb:2}
\end{table}%

\begin{table}[htb]
\caption{Anisotropic KV-CBO method for Rastrigin ($\Delta t = 0.05$, $\sigma=10$) and XSY random (($\Delta t = 0.01$, $\sigma=5$)) functions  with $\alpha=5\times 10^7$.
}
\begin{center}
{\footnotesize
\begin{tabular}{lccc|lccc}
 \multicolumn{8}{c}{}\\
\hline
\hline
 {\rm Rastrigin} & $N=50$ & $N=100$ & $N=200$ & {\rm XSY } & $N=50$ & $N=100$ & $N=200$\\
                & $M=30$ & $M=60$ & $M=120$ &{\rm random }& $M=30$ & $M=60$ & $M=120$\\
           
\hline
Rate & 99\% & 100\% & 100\% & Rate & 100\% & 100\% & 100\%\\ 
Error & 1.40e-02 & 1.08e-02 & 8.03e-03 &Error & 3.91e-02 & 3.99e-02 & 3.85e-02\\
$N_{avg}$ & 36.9 & 65.3 & 129.6 &$N_{avg}$ & 12.0 & 14.1 & 17.9\\
$n_{avg}$ & 3024.6 & 2819.2 & 3063.7&$n_{avg}$ & 15112.4 & 14519.2 & 14753.4\\
\hline
\hline
\end{tabular}
}
\end{center}
\label{tb:2b}
\end{table}%

\subsubsection{Robust PCA for Synthetic Data}\label{sec:reallife}

In this section we investigate the KV method for robust PCA of a centered point cloud $\mathcal Q= \{x^{(i)} \in \RR^d: i=1,...,\mathcal{P} \}$ in a Euclidean space. Our robust PCA method \cite{lerman15,Lerman_2017,lerman19} consists of the following two steps: 1. definition of an energy function that does not weight outlying data points too heavily, and 2. minimization of this energy function with the anisotropic KV method. For the first task we set
\begin{equation}
\EE_p(v):=\sum_{i=1}^{\mathcal{P}} | (I-v \otimes v)x^{(i)}|^p =\sum_{i=1}^{\mathcal{P}} \big( |x^{(i)} |^{2}-|\langle x^{(i)}, v \rangle |^{2}\big)^{p/2}, \quad v \in \SS^{d-1}, \label{Energy}
\end{equation}
which for $0<p<2$ is a difficult non-convex optimization problem, see \cite{FHPS2} for details. We consider a synthetic point cloud with \textit{cell-wise} and \textit{case-wise contamination} generated by the \textit{Haystack model}, see \cite{lerman15}. More precisely, we chose a vector $w \in \SS^{d-1}$ uniformly at random and then sample the \textit{inliers} from a Gaussian with rank-1 covariance matrix $\boldsymbol{\Sigma}_{in}= w \otimes w$. We then perturb these \textit{inliers} by adding Gaussian noise within the ambient space $\mathbb{R}^d$ representing the \textit{cell-wise contamination} of the point cloud. In summary, the inliers are sampled as 
\begin{equation}
x_{in}^{(i)} \sim \mathcal{N}(\textbf{0},  \boldsymbol{\Sigma}_{in} + 10^{-4}\textbf{I}_d)
\end{equation}
for $i = 1,..., \mathcal{P}_{in}$. Note, that these samples remain very close to the one-dimensional subspace $\operatorname{span}\{w\}$ we wish to detect and therefore do not constitute outliers by any means. Indeed the distance of the samples to the subspace is $10^{-4}d$ in expectation.

\noindent
To make the problem more difficult we add \textit{outliers} or \textit{case-wise contamination}, that is, data points that could be sampled from any other distribution, say, a Gaussian with a different covariance matrix. Here we sample the outliers as
\begin{equation}
x_{out}^{(i)} \sim \mathcal{N}(\textbf{0},  \boldsymbol{\Sigma}_{out}) 
\end{equation}
for $i=1,...,\mathcal{P}_{out}$, where we choose $ \boldsymbol{\Sigma}_{out} = \textbf{I}_{d}/d$. We scale the covariance matrix as proposed to achieve $E|x^{(i)}_{in}|^2 \approx E|x^{(i)}_{out}|^2$ which annihilates the option of screening for outliers by looking at the norm of the data points. Next, we note that the minimizer of $\EE_p$ is, in general, not equal to the direction $w$ we used to sample the inliers $x_{in}^{(i)}$; the proposed energy function assigns a lower weight to the ouliers $x_{out}^{(i)}$ than the standard SVD energy, but the weight will not be zero. 

\noindent
For our numerical experiment we fix $p=1$ and the total number of points to $\mathcal{P}=\mathcal{P}_{in} + \mathcal{P}_{out}=200$ and chose the number of outliers $\mathcal{P}_{out}$ as a certain percentage of $\mathcal{P}$ ranging from $5\%$ to $95\%$. 
In Table \ref{Table Outliers KV2} we compared the error to the noiseless SVD (no outliers) solution of the KV methods in dimension $d=100$ with a version of it, which we call \textit{Gradient-KV method} (GKV), that implements also gradient steps (see Section \ref{gradKV} below), and the state of the art algorithm Fast Median Subspace (FMS) \cite{Lerman_2017}, which is based on an iteratively re-weighted minimization very much in the spirit of a quasi-Newton method. Despite the fact that the KV method is derivative-free and the objective function $\EE_1$ is non-smooth and highly non-convex, with a proper tuning of the problem dependent parameters $\sigma, \Delta t$, its performances in terms of accuracy are comparable with the GKV and FMS, which do use some gradient information. When the algorithm is not fed with appropriate parameters, then it may fail to obtain high accuracy as it is shown in Table \ref{Table Outliers KV1}. However, a small modification of the KV method to include ``parsimonious'' gradient information as in GKV returns to solve the reconstruction problem with a high accuracy. 

\subsubsection*{Gradient-KV method}\label{gradKV}

The standard KV method is a zero order method that does not evaluate the tangential gradient of the cost function $\nabla_{\SS^{d-1}} \EE_p$. The modification that we propose reads as follows: every $\ell$-th iteration of the KV method we randomly choose one agent with which we perform a gradient descent step where the step size is chosen with a backtracking line-search which we iterate until the \textit{Armijo condition} is satisfied, see \cite{AMS08a}. We injection of gradient information is parsimonious, as we do not compute the gradient for every iteration and for every agent, but rather sparsely iteration-wise and agent-wise To be practical, the parameter $\ell$ might by chosen, for instance, as $\ell=10$. 

 \begin{algorithm}
\caption{Gradient KV (single iteration) 
\label{Gradient KV (single iteration)}}
{\small \begin{algorithmic}
\STATE{Randomly choose one agent $V_n^j$}
\STATE{Compute the tangential gradient $\nabla_{\mathbb{S}^{d-1}} \EE_p(V_n^j)$}
\STATE{Perform a line search to find an appropriate step size $h_j$}
\STATE{Update $V_n^j$ with gradient descent
{\small \begin{equation}
\tilde{V}_n^j \gets V_n^j - h_j \nabla_{\mathbb{S}^{d-1}} \EE_p(V_n^j), \quad V_n^j \gets \tilde{V}_n^j / |\tilde{V}_n^j|. \notag
\end{equation}}}
\STATE{Continue with the (standard) KV method (with all the agents)}
\end{algorithmic}}
\end{algorithm}

Next, it is clear that a fixed step size $h_n=h$ for all iterations $n$ does not make sense for complex non-convex objective functions. Instead we use a backtracking line search method to find an appropriate step size $h_n$. The basic idea is the following: the optimal step size $h^\star_n$ in the gradient descent method is given by 
\begin{equation}
h^\star_n = \arg \min_{h>0} \phi(h), \quad \phi(h) = \EE_p(V^j_n - h\nabla_{\SS^{d-1}} f(V^j_n)). \notag
\end{equation}
Since this optimization problem is not quickly solvable in general, we rely on heuristic methods to find an appropriate step size $h_n$. We start with an initial step size $h_n^0=1$ and check whether the \textit{sufficient decrease condition} or \textit{Armijo condition}
\begin{equation}
\EE_p(V^j_n - h_n^0\nabla_{\SS^{d-1}} \EE_p(V^j_n)) \leq \EE_p(V^j_n) -c h_n^0 |\nabla_{\SS^{d-1}} \EE_p(V^j_n)|^2 \notag
\end{equation}
is satisfied, where we chose $c=10^{-4}$. If the condition is satisfied we set $h_n = h_n^0$, otherwise we set $h_n^{1}=\tau h^0$ with, say, $\tau = 1/2$ and check again whether the \textit{sufficient decrease condition} is satisfied.\\

\begin{table}[htb]
		\begin{center}
		\caption{Error to noiseless SVD solution (no outliers). Numerical comparison of the anisotropic KV and Gradient-KV method in dimension $d=100$ with a point cloud generated by the \textit{Haystack model}. We chose $p=1$. The total number of points was $\mathcal{P}=200$. We chose tuned parameters for the KV methods, namely, $N=100$, $M=50$, $\sigma=1$, $\Delta t = 0.5$ and $T=1000$. }
		{\footnotesize
		\begin{tabular}{c|cccccccc}
			Outliers & & $5\%$ & $25\%$ & $50\%$ & $75\%$ & $95\%$ \\
			\hline
			\hline
			KV & & 6.10e-03 & 3.57e-03 & 4.49e-03 & 4.39e-03 & 7.09e-03 \\
			GKV & & 7.53e-04 & 8.21e-04 & 1.01e-03 & 1.44e-03 & 3.95e-03 \\
			FMS & & 6.85e-04 & 8.27e-04 & 1.04e-03 & 1.51e-03 & 3.82e-03 \\
			\hline
			\hline
		\end{tabular}
		}
		\label{Table Outliers KV2}
	\end{center}
\end{table}%

\begin{table}[htb]
	\begin{center}
	\caption{Same experiment as above with a generic time step $\Delta t = 0.05$ instead of the perhaps unusually larger $\Delta t = 0.5$. The results for FMS are exactly the same as in the table above. The performance of the standard anisotropic KV method depends on the tuning of parameters.}
	{\footnotesize
		\begin{tabular}{c|cccccccc}
			Outliers & & $5\%$ & $25\%$ & $50\%$ & $75\%$ & $95\%$ \\
			\hline
			\hline
			KV & & 5.86e-01 & 6.04e-01 & 5.84e-01 & 5.69e-01 & 6.95e-01 \\
			 GKV & & 7.42e-04 & 8.07e-04 & 1.05e-03 & 1.47e-03 & 3.68e-03 \\
			 FMS & & 6.85e-04 & 8.27e-04 & 1.04e-03 & 1.51e-03 & 3.82e-03 \\
			\hline
			\hline
		\end{tabular}
		}
				\label{Table Outliers KV1}
		\end{center}
\end{table}

\subsubsection{Robust computation of eigenfaces}

In this section we discuss the numerical results of the anisotropic KV on real-life data. The setup is the same as in \cite{FHPS2}: we chose a subset of $\mathcal{P}=421$ similar looking pictures of the \textit{10K US Adult Faces Database}, \cite{10kUSAdultFaces} of size $64 \times 45$, which yields a point cloud $X \in \mathbb{R}^{2880\times 421}$. We then add $6$ and $12$ outliers (pictures of animals and plants on a white background). 
We compare the results of the isotropic KV, the anisotropic KV, the anisotropic GKV, and FMS. 
We quantify the quality of the eigenface with the Peak Signal-to-Noise ratio, see Table \ref{tb:ps2n}. With such as small number of agents, i.e., $N=500$, the isotropic KV fails to perform a reasonable reconstruction with a visibly poor result. In our previous work \cite{FHPS2} we showed that this method would succeed with high accuracy if one would use at least $N=2500$ agents. Instead the anisotropic KV and GKV show high accuracy result already with a moderate number of agents and anisotropic GKV turns out to be significantly faster than the anisotropic KV.
Finally, let us stress that the eigenface from Figure \ref{fig:faces 6 outliers} and Figure \ref{fig:faces 12 outliers}  is clearly not located at cardinal positions, for which the anisotropic algorithm is expected to work best. In fact it is not even a component-wise sparse image.

 \begin{figure}[htb]
\begin{minipage}{1\textwidth}
\begin{minipage}{0.19\textwidth}
			\includegraphics[height = 2.5cm]{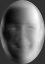} 	\centering
			{\\ (a) SVD \\ no outliers}
		\end{minipage}
		\begin{minipage}{0.19\textwidth}
			\includegraphics[height = 2.5cm]{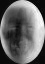} 	\centering
			{\\ (b) SVD \\ with outliers}
		\end{minipage}
		\begin{minipage}{0.19\textwidth}
			\includegraphics[height = 2.5cm]{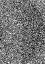}    \centering
			{\\ (c) isotropic KV $N=500$}
		\end{minipage}
		\begin{minipage}{0.19\textwidth}
			\includegraphics[height = 2.5cm]{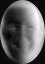}    \centering
			{\\ (d) anisotropic KV $N=500$}
		\end{minipage}
		\begin{minipage}{0.19\textwidth}
			\includegraphics[height = 2.5cm]{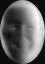}    \centering
			{\\ (e) anisotropic GKV $N=500$}
		\end{minipage}
	\end{minipage}
	\caption{Eigenfaces of the point cloud with no outliers computed by SVD (a), with outliers by SVD (b), isotropic KV (c), anisotropic KV (d),  and anisotropic Gradient-KV (e). The batch size is always chosen as $M=10\% \cdot N$. We have used: $p=1$, $\alpha = 10^{5}, \Delta t = 0.25, n_T = 10^5, \mu = 0$ and $\sigma = 0.02$ (isotropic noise) and $\sigma = 1$ (anisotropic noise).}\label{fig:faces 6 outliers}
\end{figure}

\begin{figure}[htb]
\begin{minipage}{1\textwidth}
\begin{minipage}{0.19\textwidth}
			\includegraphics[height = 2.5cm]{SVD.png} 	\centering
			{\\ (a) SVD \\ no outliers}
		\end{minipage}
		\begin{minipage}{0.19\textwidth}
			\includegraphics[height = 2.5cm]{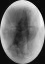} 	\centering
			{\\ (b) SVD\\ with outliers}
		\end{minipage}
		\begin{minipage}{0.19\textwidth}
			\includegraphics[height = 2.5cm]{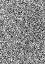}    \centering
			{\\ (c) isotropic KV $N=500$}
		\end{minipage}
		\begin{minipage}{0.19\textwidth}
			\includegraphics[height = 2.5cm]{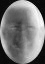}    \centering
			{\\ (d) anisotropic KV $N=500$}
		\end{minipage}
		\begin{minipage}{0.19\textwidth}
			\includegraphics[height = 2.5cm]{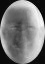}    \centering
			{\\ (e) anisotropic GKV $N=500$}
		\end{minipage}
		\hfill
	\end{minipage}
	\caption{Eigenfaces for a point cloud with 12 outliers. We chose $p=0.5$, the other parameters are the same as in Figure \ref{fig:faces 6 outliers}.}\label{fig:faces 12 outliers}
\end{figure}

\begin{table}[htb]
		\caption{Peak Signal-to-Noise Ratios for the eigenfaces from above. The reference image is always Figure \ref{fig:faces 6 outliers} (a), that is, the eigenface for the point cloud with no outliers computed by SVD. We note that for the point cloud with 6 outliers (Figure \ref{fig:faces 6 outliers}) the standard anisotropic KV with $N=500$ agents produced the best eigenface in terms of PS2N ratio; the computation took around 6 hours on a standard 2.6 GHz processor. The computation of the corresponding eigenface with the Gradient-KV method in column (e) took around 20 minutes. Note that the eigenface computed by FMS is not displayed in Figures \ref{fig:faces 6 outliers} - \ref{fig:faces 12 outliers} due to space limitations. \label{tb:ps2n}}

	\begin{center}
	{\footnotesize
		\begin{tabular}{c|ccccccc}
			 & (b) & (c) & (d) & (e) & FMS \\
			\hline
			\hline
			 Figure \ref{fig:faces 6 outliers} & $15.98$ & $9.36$ & ${\bf 21.05}$ &  $20.68$ & $20.68$ \\
			\hline
			 Figure \ref{fig:faces 12 outliers} & $12.31$ & $9.16$ & ${\bf 14.78}$ &   $14.29$ & $14.28$ \\
			\hline
			\hline
		\end{tabular}
	}
	\end{center}
\end{table}%

\subsubsection{The Phase Retrieval Problem}\label{sec:phaseretr}

We consider the phase retrieval problem from quadratic measurements in $\mathbb{R}^d$: Reconstruct $z^*\in \mathbb R^d$ from measurements of the form
\begin{equation}
y_i = |\langle v^*, a_i\rangle |^2+ w_i, \quad i=1,...,M\,, \label{def yi}
\end{equation}
where $w_i$ is adversarial noise, and $a_i$ are a set of known vectors. That is, we measure only the (squared) magnitude of $\langle v^*,a_i\rangle$, and not the phase (or the sign, in the case of real valued vectors). We solve problem \eqref{def yi} in the noiseless case, i.e., $w=0$, by empirical risk minimization. As discussed in \cite{FHPS2} the unconstrained empirical risk minimization can be recast without loss of generality as a constrained optimization problem on the sphere once the lower frame bound $A$ of $\{a_i\}_{i=1}^{M}$ is known., i.e., we aim at minimizing
\begin{equation}
\EE(v) := \frac{1}{\mathcal M}\sum_{i=1}^{\mathcal{M}} \left | |\langle v, a_i \rangle|^2 - y_i \right |^2, \label{empiricalrisk2}
\end{equation}
over the sphere $\mathbb S^{d}$. 

\begin{figure}[htb]
	\centering
	\includegraphics[width = 6cm]{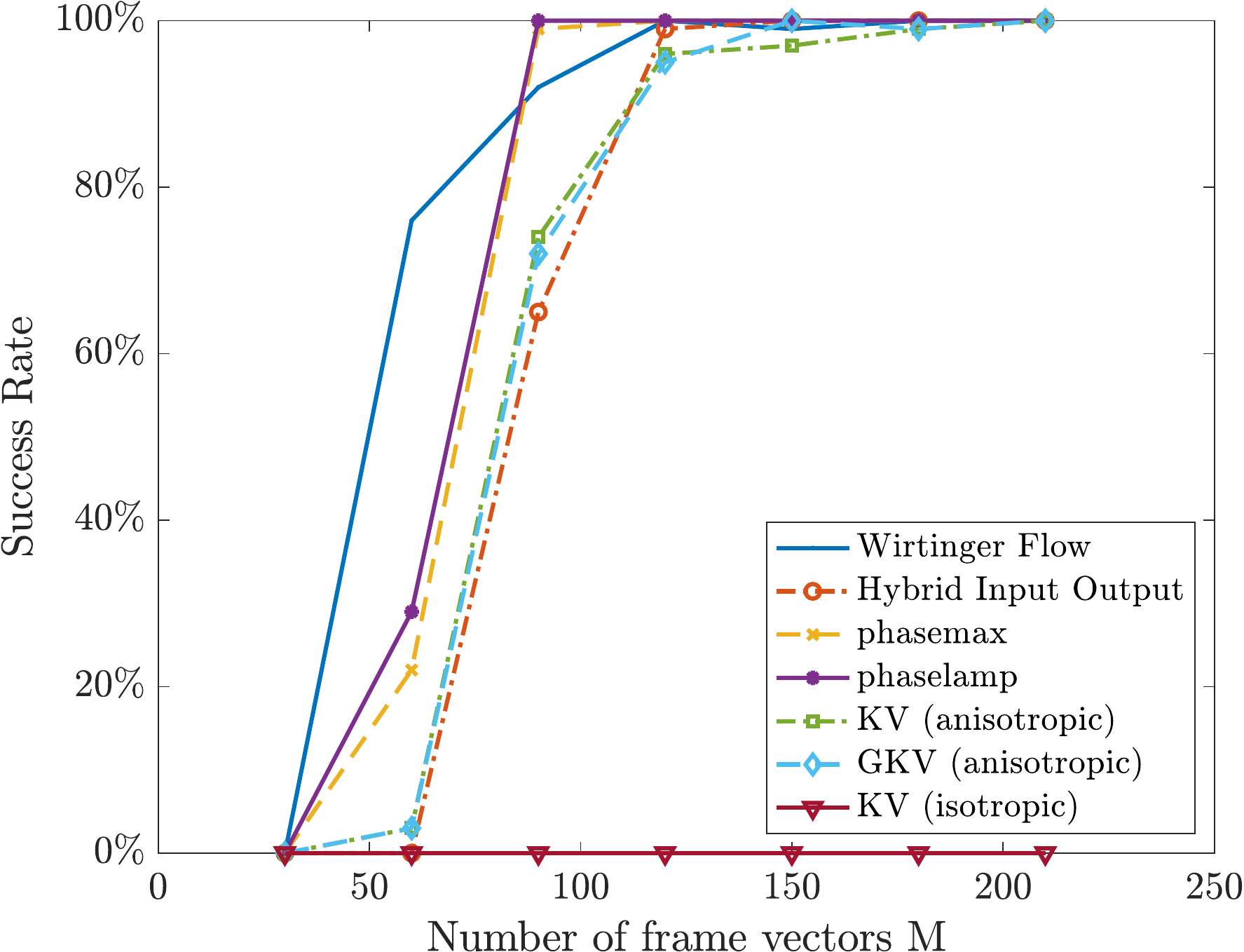}
	\includegraphics[width = 6cm]{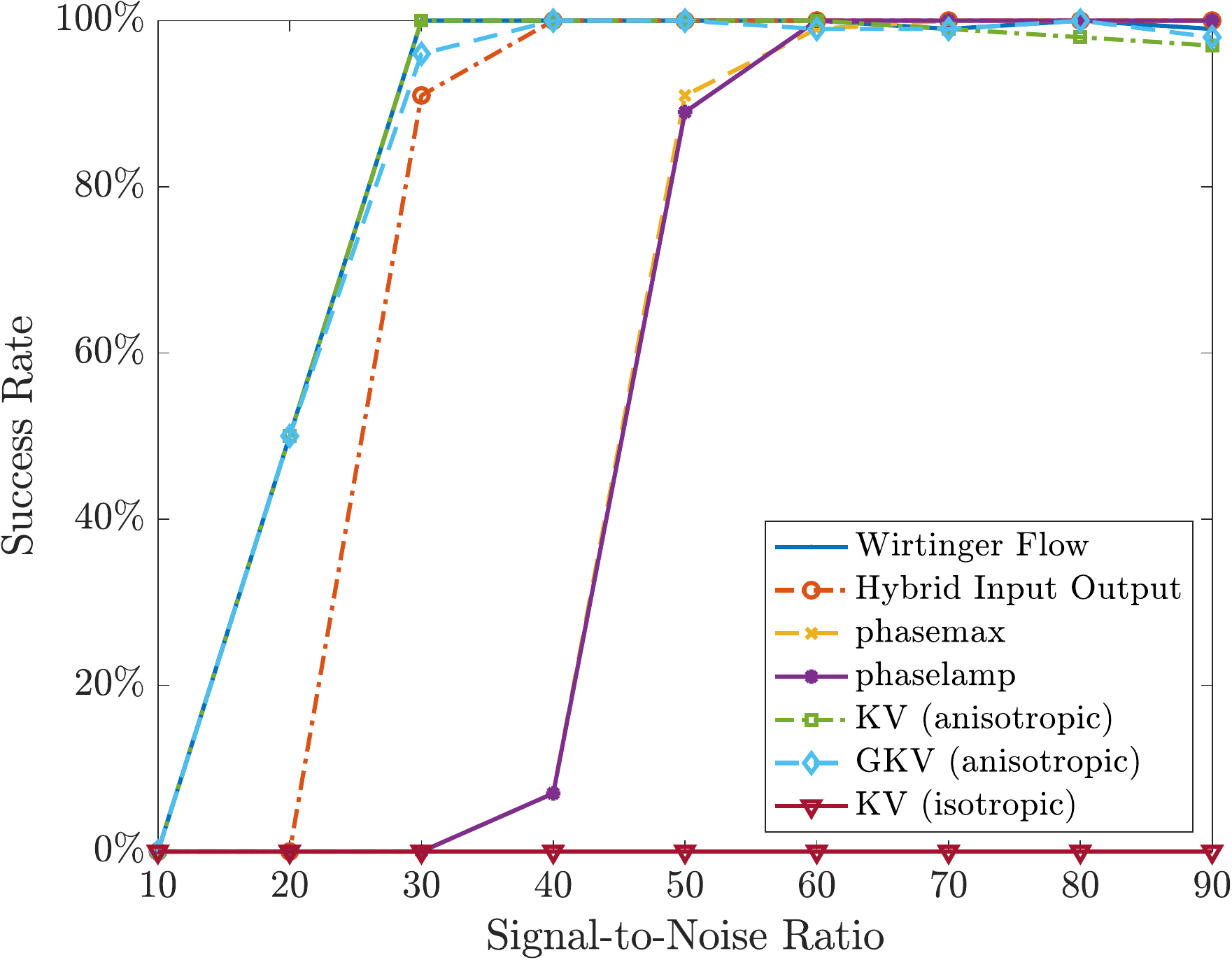}
	\caption{Success rate in terms on number of frame vectors $M$ (left) and Signal-to-Noise ratio (right) for a Gaussian frame in dimension $d=30$. We have used: $N=500$, $M=50$, $\Delta t = 0.5$, $\alpha = \infty$ and $\sigma = 1$ (anisotropic) resp. $\sigma = 0.2$ (isotropic). We have further used $T=2000$ (left) and $T=5000$ (right). The results are averaged over $100$ runs. We note that the standard isotropic KV fails to reconstruct the signal with $N=500$ agents. This is consistent with our findings in the robust computation of eigenfaces, see Figures \ref{fig:faces 6 outliers} - \ref{fig:faces 12 outliers}. The isotropic KV method proved successful in \cite{FHPS2} with at least $N=10^4$ agents.}\label{fig:pr}
\end{figure}

In Figure \ref{fig:pr} we compare Algorithm \ref{algo:KV-CBOfc} with its isotropic version and three relevant state of the art methods for phase retrieval, namely Wirtinger Flow (fast gradient descent method) \cite{ca14,Chen_2019}, Hybrid Input Output/Gerchberg-Saxton's Alternating Projections (alternating projection methods) \cite{Gerchberg72,fien82,Yang:94} and PhaseMax/PhaseLamp (convex relaxation and its multiple iteration version) \cite{ca13}. For the comparsion we used the Matlab toolbox PhasePack\footnote{\it https://www.cs.umd.edu/$\sim$tomg/projects/phasepack/} \cite{chandra2017phasepack} and our own code\footnote{{\it 
		https://github.com/PhilippeSu/KV-CBO}}. The numerical experiments show that the anisotropic KV is significantly superior with respect to its isotropic version and it is able to perform nearly as gradient based state-of-the-art methods such as Wirtinger Flow (which is actually a gradient flow).\\

\section{Conclusion}

We presented  a new consensus-based model for global optimization on the sphere, with an anisotropic random exploration term.
The main result of this paper is about {the proof of the convergence} provided conditions of well-preparation of the initial datum. We presented also several numerical experiments in low dimension and synthetic examples in order to illustrate the behavior of the method and we tested the algorithms in high dimension against state of the art methods in a couple of challenging problems in signal processing and machine learning.
When it comes to computing minimizers near cardinal positions, we documented  the tremendous advantage  of the anisotropic scheme \eqref{Intro KViso num3} over its isotropic counterpart \eqref{Intro KViso num}  in synthetic numerical experiments for the optimization of very challenging test functions. 
Despite the evidence that the advantage of the anisotropic method is particularly efficient in high-dimension for minimizers near cardinal points, we also show in Section \ref{sec:reallife} that the anisotropic scheme \eqref{Intro KViso num3} significantly outperforms the isotropic one \eqref{Intro KViso num} in real-life applications in high-dimension, namely in phase retrieval problems and in robust linear regression. In these applications there is no guarantee that minimizers are near cardinal positions. Hence, these real-life experiments suggest that, despite the anisotropy we introduced is  dependent on the embedding of the sphere in the Euclidean space, the anisotropic numerical scheme should be the preferred choice in practice and it is extremely efficient also in high-dimension.

%
\section*{Acknowledgments} MF and HH acknowledge the support of the DFG Project "Identification of Energies from Observation of Evolutions".     The present project and PS  are supported by the National Research Fund, Luxembourg (AFR PhD Project Idea ``Mathematical Analysis of Training Neural Networks'' 12434809). 
LP acknowledges the support of the John Von Neumann guest Professorship program of the Technical University of Munich during the preparation of this work. 
	The authors acknowledge the support and the facilities of the LRZ Compute Cloud of the Leibniz Supercomputing Center of the Bavarian Academy of Sciences, on which the numerical experiments of this paper have been tested.

\bibliographystyle{plain}
\bibliography{references}

\end{document}